\documentclass[12pt]{amsart}
\usepackage{enumerate}
\usepackage{amssymb}
\usepackage{graphicx}
\usepackage{amscd, color}
\usepackage{amsmath}
\usepackage{amsfonts}
\usepackage[english]{babel}
\usepackage{epsfig}
\usepackage{latexsym,amssymb,amsmath}
\usepackage{mathrsfs}
\usepackage{hyperref}
\usepackage{amsthm}
\usepackage{soul}

\usepackage{anysize}
\marginsize{3.3cm}{3.3cm}{3cm}{3cm}

\usepackage{float}
\usepackage[T1]{fontenc}

\usepackage{fancyvrb}
\usepackage{alltt}
\usepackage{fancyhdr}
\usepackage[parfill]{parskip}
\usepackage[shortlabels]{enumitem}
\usepackage{afterpage}
\usepackage{setspace}
\usepackage[all,cmtip]{xy}
\usepackage[bottom]{footmisc}
\usepackage{xr}
\usepackage{pgfplots}
\pgfplotsset{compat=1.11}
\usepackage{tikz}
\usetikzlibrary{shapes,arrows}
\usetikzlibrary{shapes.geometric}
\usetikzlibrary {shapes.symbols}
\usetikzlibrary{positioning}
\usetikzlibrary{shadings}
\usetikzlibrary{patterns}
\usetikzlibrary{decorations.pathreplacing,calligraphy, decorations.pathmorphing}

\numberwithin{equation}{section}
\newtheorem{theorem}{\textit{Theorem}}[section]
\newtheorem{proposition}[theorem]{\textit{Proposition}}
\newtheorem{definition}[theorem]{\textit{Definition}}
\newtheorem{lemma}[theorem]{\textit{Lemma}}
\newtheorem{corollary}[theorem]{\textit{Corollary}}

\newtheorem{remark}[theorem]{\textit{Remark}}
\newtheorem{example}[theorem]{\textit{Example}}

\DeclareMathOperator{\cco}{\overline{conv}}

\title{New types of convergence for unbounded star-shaped sets}
\author{Luisa F. Higueras-Montaño}
\address{Departamento de  Matem\'aticas,
Facultad de Ciencias, Universidad Nacional Aut\'onoma de M\'exico, 04510 Ciudad de M\'exico, M\'exico.}

\email{luisaf.hm@ciencias.unam.mx}

\keywords{Star-shaped sets, Wijsman convergence, Attouch-Wets convergence, Duality, Flowers, Convex sets}

\subjclass[2020]{Primary: 52A20, 52A30, 54A10, 54B20. Secondary: 40A30, 54C10}

\begin{document}

\begin{abstract}
    We introduce radial variants of the Wijsman and Attouch-Wets topologies for the family $\mathcal{S}_{rc}^d$ of star sets $A \subseteq \mathbb{R}^d$ that are radially closed.
    These topologies give rise to new types of convergence for star-shaped sets with respect to the origin, even when such sets are not closed or bounded. Our approach relies on a new family of functionals, called \textit{radial distance functionals}, which measure ``radial distances'' between points $x \in \mathbb{R}^d$ and sets $A \in \mathcal{S}_{rc}^d$. These are natural radial analogues of the classical distance functionals.
    We prove that our radial Wijsman type topology $\tau_{W^r}$ is not metrizable on $\mathcal{S}_{rc}^d$, while our radial Attouch-Wets type topology $\tau_{AW^r}$  is completely metrizable. A corresponding radial Attouch-Wets distance $d_{AW^r}$ is introduced, and we prove that $d_{AW}(A,K) \leq d_{AW^r}(A,K)$ for all closed $A,K \in \mathcal{S}_{rc}^d$, where $d_{AW}$ denotes the Attouch-Wets distance. Among others,
    these results are applied to prove the continuity of the star duality on $\mathcal{S}_{rc}^d$ with respect to both $\tau_{W^r}$ and $\tau_{AW^r}$, and to establish topological properties of the family of flowers associated with closed convex sets containing the origin.
    \end{abstract}

\maketitle

\section{Introduction}

The geometry of star-shaped sets in Euclidean spaces has been an active and ongoing area of research since the 1960's. It is well known that these sets play a fundamental role in the dual Brunn-Minkowski theory and are frequently used in fields such as
functional analysis, operations research, and computational geometry, among others. A comprehensive overview of the literature and current research directions can be found in the remarkable survey \cite{Moszynska2020}.

Let $\mathbb{R}^d$, with $d\geq2$, be the  $d$-dimensional Euclidean space equipped with standard norm $\|\cdot\|$ and scalar product $\langle\cdot,\cdot\rangle$. The corresponding unit ball and unit sphere are denoted by $B_2^d$ and $\mathbb{S}^{d-1}$, respectively. In this work, we are particularly interested in star-shaped sets with respect to the origin. We allow our sets being unbounded or non-closed. Accordingly, by a \textit{star set} we mean a non-empty $A\subseteq\mathbb{R}^d$ such that $ta\in A$ for every $a\in A$ and $0\leq t\leq1$. 

Associated with each star set $A$, the \textit{radial function}  $\rho_A:\mathbb{S}^{d-1}\rightarrow[0,\infty]$ is defined by
$$
\rho_A(\theta):=\sup\left\{\lambda\geq0: \lambda\theta\in A\right\}.
$$

A \textit{star body} is a star set $A\subseteq\mathbb{R}^d$ that is \textit{radially closed}, meaning  that $\rho_A(\theta)\theta\in A$ for all directions $\theta\in\mathbb{S}^{d-1}$ such that $\rho_A(\theta)<\infty$. Observe that each star body is uniquely determined by its radial function. The family of all star bodies in $\mathbb{R}^d$ will be denoted by $\mathcal{S}_{rc}^{d}$. The subfamily of star bodies $A\subseteq\mathbb{R}^d$ with continuous radial function $\rho_A$ will be denoted by $\mathcal{S}_{1}^{d}$.

Significant research has been devoted to the geometric and topological properties of star sets in $\mathbb{R}^d$, see for instance \cite{Beer1975,BeerKlee1987,GardnerBook,HansenMartini2010,HansenMartini2011,Klain96,LipschitzStarbodies,MilmanRotem2017RadialSums,Sojka2013} and the references in \cite{Moszynska2020}. However, most of this work imposes conditions such as boundedness or closedness on the sets under consideration. Here, in the spirit of the seminal work \cite{Wijsman1966} by R. Wijsman, we address the problem of providing appropriate notions of convergence  for  star sets, which may be unbounded or non-closed.

Our main contribution is the introduction of new notions of convergence for star bodies, defined through new topologies on the family $\mathcal{S}_{rc}^{d}$. These topologies are conceived as radial counterparts of the Wijsman and the Attouch–Wets topologies. We refer to them as the \textit{radial Wijsman topology} $\tau_{W^r}$ (Definition \ref{defn:radial-WijsmanTop}) and the \textit{radial Attouch-Wets topology} $\tau_{AW^r}$ (Definition \ref{def:radialAW}), respectively. 
They provide a natural framework to study the topological properties of the \textit{star duality} $\Phi:\mathcal{S}_{rc}^{d}\rightarrow\mathcal{S}_{rc}^{d}$,  defined by
\begin{equation}\label{eq:duality-Src}
\rho_{\Phi(A)}:=\frac{1}{\rho_{A}}.
\end{equation}
We adopt the convention that $\frac{1}{0}=\infty$ and $\frac{1}{\infty}=0$.  The duality $\Phi$, first introduced by M. Moszy\'nska \cite{Moszynska1999} on the family of compact star sets with continuous radial functions and the origin in the interior, was recently studied by E. Milman, V. Milman and L. Rotem \cite{MMilmanRotem} on the broader family $\mathcal{S}_{rc}^d$. There, a key connection is drawn  between $\Phi$ and the classical polar duality (see \cite[Theorem 9.12]{MMilmanRotem} or Section \ref{sec:dualities} of the present work). More recently, cost functions were used in \cite{ZooDualities} to exhibit new dualities on $\mathcal{S}_{rc}^d$.

In Theorems \ref{thm:duality-Wr-continuous} and \ref{thm:duality-awr-continua}, we prove the continuity of the duality $\Phi$ with respect to the topologies $\tau_{W^r}$ and $\tau_{AW^r}$, respectively. As a result, on the spaces $\mathcal{X}=(\mathcal{S}_{rc}^{d},\tau_{W^r})$ and $\mathcal{X}=(\mathcal{S}_{rc}^{d},\tau_{AW^r})$, the map $\Phi:\mathcal{X}\rightarrow\mathcal{X}$ is a continuous duality such that $B_2^d$ is the unique fixed point. 

A key ingredient in developing our radial variants of the Wijsman and Attouch-Wets topologies is the introduction of a \textit{radial distance functional},  a tool that measures the  ``radial distance''  between a star body and a point of $\mathbb{R}^d$. This functional is defined as follows:

\begin{definition}\label{defn:radialdist}
    Let $A\in\mathcal{S}_{rc}^{d}$ and $x\in\mathbb{R}^d$. Then the radial distance from $x$ to $A$ is  defined as
    $$
    d_r(x,A):=\inf\left\{\alpha>0:x\in A\tilde{+}\alpha B_2^d \right\}.
    $$
    The map $d_r(\cdot,A):\mathbb{R}^d\rightarrow[0,\infty)$ sending $x$ to $d_r(x,A)$ is called the radial distance functional associated to $A$.
\end{definition}

Here $\tilde{+}$ denotes the radial sum of star bodies (see Section \ref{sec:preliminaries}). The basic properties of the distance functional are established in Section \ref{sec:radial-distance-functional}. Among other results, we show that the map sending each  $A\in\mathcal{S}_{rc}^d$ to $d_r(\cdot,A)$ is injective ((P3) of Remark \ref{rem:properties-d-r}), and that distance functionals can be used as an alternative to compute the radial metric  $\delta(A_1,A_2)$ between bounded star bodies $A_1,A_2\subseteq\mathbb{R}^d$. This relation is formally established in Proposition \ref{prop:caract-dHr-distanciaradial}, where we prove that
$$
\delta(A_1,A_2)=\sup_{x\in\mathbb{R}^d}\left|d_r(x,A_1)-d_r(x,A_2)\right|.
$$
It is worth noticing that radial distance functionals may be discontinuous (see Section \ref{sec:encajes-sin-estruct}). Despite this fact, Proposition \ref{prop:car-starb-conti-distance} proves that $d_r(\cdot, A)$ is continuous if and only if $A\in\mathcal{S}_1^d$, and that each $A\in\mathcal{S}_1^d$ must be closed.

Below, we briefly outline the notions of convergence introduced in this work. The convergence on $\mathcal{S}_{rc}^d$ determined by the radial Wijsman topology $\tau_{W^r}$ corresponds to the pointwise convergence of radial distance functionals. Specifically, a net $(A_i)_{i\in I}\subseteq\mathcal{S}_{rc}^d$ converges in the topology $\tau_{W^r}$  to some $A\in\mathcal{S}_{rc}^d$ (i.e., it is $\tau_{W^r}$-convergent to $A$) if and only if
$$
\lim_{i\in I}d_r(x,A_i)=d_r(x,A),\textit{ for every }x\in\mathbb{R}^d.
$$
See (\ref{eq:pointwise-convergence-rdistance}). Section \ref{sec:WijsmanTop} focuses on the radial Wijsman topology on $\mathcal{S}_{rc}^d$ and $\mathcal{S}_{1}^d$. Among other results, we show that convergence in $\tau_{W^r}$ is equivalent to the pointwise convergence of the corresponding radial functions (Theorem \ref{thm:caract-rWijsmanTopology}). As a consequence, we establish in Corollary \ref{cor:rWijsman-nonmetrizable} that $\tau_{W^r}$ is not metrizable on $\mathcal{S}_{rc}^d$. We also investigate how this topology relates to the classical Wijsman topology on the spaces of closed star bodies and closed convex sets containing the origin; see Section \ref{sec:rel-Wr-W} for details.

The notion of convergence induced by the radial Attouch-Wets topology $\tau_{AW^r}$ is established in Section \ref{sec:radial Attouch-Wets topology}. This notion corresponds to the uniform convergence of radial distance functionals on bounded sets of $\mathbb{R}^d$. In this way, a sequence $(A_n)_n$ of star bodies is $\tau_{AW^r}$-convergent to some $A\in\mathcal{S}_{rc}^d$ if and only if
$$
d_r(\cdot,A_n)\textit{ converges  uniformly to }d_r(\cdot,A),\textit{on every bounded }C\subseteq\mathbb{R}^d.
$$
See Definition \ref{defn:radial-WijsmanTop}. This notion of convergence is described in terms of sequences, since, in Theorem~\ref{thm:rAW-complete-metrizable}, we show that $\tau_{AW^r}$ is a completely metrizable topology, in contrast with $\tau_{W^r}$. 
Furthermore, the map $d_{AW^r}:\mathcal{S}_{rc}^d\times\mathcal{S}_{rc}^d\rightarrow[0,1]$ defined by
\begin{equation}\label{defn:rd-AW}
    d_{AW^r}(A_1,A_2):=\sup_{j\in\mathbb{N}}\min\left\{\frac{1}{j}, 
\sup_{\|x\|\leq j}\left|d_{r}(x,A_1)-d_r(x,A_2)\right|\right\}
\end{equation}
is a metric compatible with $\tau_{AW^r}$, and the space $(\mathcal{S}_{rc}^d, d_{AW^r})$ is complete. We refer to the metric $d_{AW^r}$ as \textit{the radial Attouch-Wets distance}. In this case, the subspace $\mathcal{S}_{1}^d\subseteq(\mathcal{S}_{rc}^d, d_{AW^r})$ is closed and, therefore, $(\mathcal{S}_{1}^d, d_{AW^r})$ is complete (Proposition \ref{prop:subspaces-rdAW}).

In Proposition \ref{prop:rdAW-es-top-strnger-than-rdAW}, we exhibit a comparison between  the Attouch-Wets distance  $d_{AW}$ (see (\ref{eq:dAW})) and its radial version $d_{AW^r}$. We show that, for every pair of closed star bodies $A_1,A_2\subseteq\mathbb{R}^d$,
\begin{equation}
\label{eq:dAW-es-menor-rdAW}
d_{AW}(A_1,A_2)\leq d_{AW^r}(A_1,A_2).
\end{equation}
This inequality can be seen as a generalization of the well-known inequality between the Hausdorff distance and the radial metric for compact star bodies $A_1,A_2\subseteq\mathbb{R}^d$,  
\begin{equation}\label{eq:d-H-menor-rd-H}
d_H\left(A_1,A_2\right)\leq\delta\left(A_1,A_2\right).
\end{equation}
See e.g., \cite[Corollary 14.3.3]{MoszynskaBook}. 

Inequality (\ref{eq:dAW-es-menor-rdAW}) provides a basis for examining the relationship between the Attouch-Wets topology and its radial counterpart $\tau_{AW^r}$ on the family $\mathcal{K}_{0}^d$ of closed convex sets $K\subseteq\mathbb{R}^d$ containing the origin. It follows that $\tau_{AW}\subseteq\tau_{AW^r}$ on $\mathcal{K}_{0}^d$, although, as detailed in Remark \ref{rem:d_AWr de convexos}-(1), this inclusion is strict. For the subfamily of compact convex sets  $\mathcal{K}_{0,b}^d$, the metrics $\delta$ and $d_{AW^r}$ generate the same topology. In particular, on $\mathcal{K}_{(0),b}^d$, consisting of compact convex sets with the origin in their interior, $d_H$ and $d_{AW^r}$ generate the same topology, see Remark \ref{rem:d_AWr de convexos}-(2)-(3).

The motivation for this work arises from two recent papers \cite{MMilmanRotem} and \cite{NataliaLuisa2}.
The first introduces a new family of star bodies, called \textit{flowers}, and shows, via the duality $\Phi$, that every pair $(K,K^\circ)$ (where $K^\circ$ denotes de polar set of $K\in\mathcal{K}_0^d$) is uniquely determined by a flower (see \cite[Theorem 9.12]{MMilmanRotem},\cite{MilmanRotem2020}). The second one, coauthored by the author, highlights the relevance of studying the geometric and topological structure of $\mathcal{K}_0^d$ to characterize all dualities  that are topologically conjugate to the polar duality.

Building on these ideas, Section \ref{sec:dualities} examines the topological properties of the duality $\Phi$ and the family of flowers. We begin by noticing that $\Phi$ is also a duality for $\mathcal{S}_{1}^d$ (Proposition \ref{prop:duality-S1}). Then, as mentioned before, we establish continuity of $\Phi$ on $\mathcal{S}_{rc}^d$ with respect to the topologies $\tau_{W^r}$ and $\tau_{AW^r}$, respectively. We also introduce a new distance on the family of flowers
along with some observations about it, see (\ref{eq:flower-distance}). The paper concludes with some questions related to our newly proposed distance on the family of flowers and to the homeomorphism type of $(\mathcal{S}_{rc}^d,d_{AW^r})$ and its topological subspace $\mathcal{S}_{1}^d$.

\subsection{Preliminaries and notation}\label{sec:preliminaries}

We use standard notation $\overline{C}$, $\text{Int}(C)$ and $\partial C$ to denote, respectively, the closure, interior and boundary of a subset $C\subseteq\mathbb{R}^d$. By $\overline{\text{conv}}(C)$ we denote the closed convex hull, and by $\text{Aff}(C)$ the affine space generated by $C$. As usual, for $C_1,C_2\subseteq\mathbb{R}^d$ and $\lambda\in\mathbb{R}$, we define
\begin{align*}
    C_1+C_2:=&\left\{c_1+c_2:c_1\in C_1, c_2\in C_2\right\},\\
    \lambda C_1:=&\left\{\lambda c_1:c_1\in C_1\right\}.
\end{align*}
For $x,y\in\mathbb{R}^d$,  we write $[x,y]$ to denote the closed segment joining $x$ and $y$. In addition, for each $x\neq0$, the unitary vector $\frac{x}{\|x\|}$ is denoted by $\theta_x$. By $e_1,\ldots,e_d$ we denote the canonical orthonormal basis of $\mathbb{R}^d$. The vector space spanned by $a\in\mathbb{R}^d$ is denoted by $\langle a\rangle$. We will introduce additional notation as needed, but for basic definitions and conventions in convexity we follow \cite{Schneider2014}.

Recall that the family of star bodies in $\mathbb{R}^d$ is denoted by $\mathcal{S}_{rc}^{d}$. We write $\mathcal{S}_{rc, b}^{d}$ for the family of bounded star bodies and $\mathcal{S}_{rc, b,(0)}^{d}$ for those elements in $\mathcal{S}_{rc, b}^{d}$ containing the origin  in the interior. We emphasize that closed star sets $C\subseteq\mathbb{R}^d$ are always star bodies;
however, star bodies are not necessarily closed. For example,  for every open half-space $H\subset\mathbb{R}^d$ with the origin as a boundary point, $H\cup\{0\}$ is a non-closed star body. 

For $A_1,A_2\in\mathcal{S}_{rc}^{d}$, the \textit{radial sum} $A_1\tilde{+}A_2$ is the star body with radial function $\rho_{A_1\tilde{+}A_2}:=\rho_{A_1}+\rho_{A_2}$. Moreover, for every $x,y\in\mathbb{R}^d$, let $x\tilde{+}y$ be defined as
\begin{equation*}
    x\tilde{+}y:=\begin{cases}
       x+y &\text{ if } 0,x,y \text{ are collinear},\\
       0 &\text{ otherwise}.
       \end{cases}
\end{equation*}
Then, it can be directly verified that
$
A_1\tilde{+}A_2=\left\{a_1\tilde{+}a_2:a_1\in A_1, a_2\in A_2\right\}
$ 
and that 
$
A_1\tilde{+}A_2\subseteq A_1+A_2.
$ 
Additionally, for each $\lambda>0$ and $A\in\mathcal{S}_{rc}^{d}$, 
$\lambda A$
is a star body with $\rho_{\lambda A_1}=\lambda\rho_{A_1}$.
See for instance \cite[$\mathsection$ 14.1]{MoszynskaBook}.

On $\mathcal{S}_{rc, b}^{d}$, we are interested in the metric topology determined by the radial metric (also known as the radial Hausdorff distance \cite{Lutwak88}). For every pair $A_1,A_2\in\mathcal{S}_{rc, b}^{d}$,  the \textit{radial metric} $\delta(A_1,A_2)$ is given by any of the following expressions:
\begin{align}\label{defn:rHaudsdorff}
\delta(A_1,A_2)&=\sup_{\theta\in\mathbb{S}^{d-1}}|\rho_{A_1}(\theta)-\rho_{A_2}(\theta)|
\\\nonumber
&=\inf\left\{
\alpha>0:A_1\subseteq A_2\tilde{+}\alpha B_2^d\text{ and }A_2\subseteq A_1\tilde{+}\alpha B_2^d
\right\}.
\end{align}
For the basics results regarding star sets and the radial metric, we refer the reader to the books \cite{GardnerBook}, \cite[Chapter 14]{MoszynskaBook} and \cite[$\mathsection$ 9.3]{Schneider2014}, and the survey \cite{Moszynska2020}.

Given a non-empty closed set $C\subseteq\mathbb{R}^d$, 
its associated \textit{distance functional} $d(\cdot, C):\mathbb{R}^d\rightarrow[0,\infty)$ is  given by 
\begin{equation}\label{eq:distance-funct}
    d(x,C):=\inf_{c\in C}\|c-x\|. 
\end{equation}

For any pair $C_1,C_2$ of non-empty closed sets in $\mathbb{R}^d$, the \textit{excess} of $C_1$ \textit{over} $C_2$ is given by
$$
e(C_1,C_2):=\sup\left\{
d(x,C_2):x\in C_1
\right\}.
$$

The \textit{Hausdorff distance} between non-empty compact sets  $C,D\subseteq\mathbb{R}^d$ is given by any of the equivalent expressions below:
\begin{align*}
d_H(C,D)&=\max\left\{\sup_{x\in C}d(x,D), \sup_{x\in D}d(x,C)\right\}\\
&=\text{inf}\left\{\lambda>0:C\subseteq D+\lambda\mathbb{B},\ D\subseteq C+\lambda\mathbb{B}\right\}\\
&=\sup_{x\in\mathbb R^n}|d(x, C)-d(x,D)|
\end{align*}
See for instance \cite[\S 3.2]{Beer1993}. 

The following example, which appears in the proof of \cite[Theorem 14.3.4]{MoszynskaBook}, will be used frequently throughout this work. Therefore, we outline its key properties.

\begin{example}\label{exam:ejemploMoszynska}
    Let $L$ denote the vector space spanned by the canonical vector $e_d\in\mathbb{R}^d$ and let $a=\frac{1}{4}e_d$. For every $n\in\mathbb{N}$, let $S_n\subset\mathbb{S}^{d-1}$ be the $(d-2)$-dimensional sphere with center in $\left(1-\frac{1}{n+1}\right)e_d$ and such that the affine hyperplane $\text{Aff}(S_n)$ is orthogonal to $L$. 
    
    For every $n\in\mathbb{N}$, let $C_n$ be the cone with vertex $a$ containing $S_n$. Then the sequence of bounded star bodies $(A_n)_n\subseteq\mathcal{S}_{rc,b}^d$, defined by
    $$
    A_n=\overline{B_2^d\setminus\text{conv}(C_n)},
    $$
    satisfies the following properties:
    \begin{itemize}
        \item The radial function $\rho_{A_n}$ is continuous for every $n\in\mathbb{N}$.
        \item $(A_n)_n$ converges in the Hausdorff distance to $B_2^d$.
        \item $(A_n)_n$ does not converge in the radial metric to $B_2^d$. 
    \end{itemize}
\end{example}


Since we will work with unbounded star bodies and their associated radial maps, it is convenient to recall some basic definitions regarding maps taking values in the \textit{extended real line} $\overline{\mathbb{R}}:=[-\infty,\infty]$. 

Let $X\subseteq\mathbb{R}^d$ be a non-empty topological subspace, and let
$f:X\rightarrow\overline{\mathbb{R}}$ be a map. Then, $f$ is called \textit{proper} if it is somewhere finite and its values lie in $(-\infty,\infty]$.  The
\textit{effective domain}, $\text{dom} f$, is the set of points at which $f$ is finite. Therefore,
$$
\text{dom} f:=\left\{ x\in X: f(x)\in\mathbb{R}\right\}.
$$
Notice that radial maps $\rho_A$ associated to star bodies $A\neq\mathbb{R}^d$ are always proper. Particularly, $\text{dom} \rho_A\neq\emptyset$. 

A map $f:X\rightarrow\overline{\mathbb{R}}$ is called \textit{lower semicontinuous at}  $x\in X$ if for every sequence $(x_n)_{n}\subseteq X$ converging to $x$,
$$
\liminf_{n}f(x_n)\geq f(x).
$$  
In addition, $f$ is called \textit{lower semicontinuous} (simply, \textit{lsc}) if it is lower semicontinuous at every $x\in X$. Similarly, $f$ is called \textit{upper semicontinuous at} $x\in X$ if $\limsup_{n}f(x_n)\leq f(x)$ for every sequence $(x_n)_{n}\subseteq X$ converging to $x$. In this case, $f$ is called \textit{upper semicontinuous} (simply, \textit{usc}) whenever it is upper semicontinuous at every $x\in X$.

It  follows directly from the above definitions that a map $f:X\rightarrow\overline{\mathbb{R}}$ is \textit{continuous at} $x\in X$ if it is both lower semicontinuous and upper semicontinuous at $x$.  Consequently, $f$ is called \textit{continuous} whenever it is continuous at every $x\in X$. 

The preceding definitions concerning the continuity of extended real-valued functions are based on \cite{Beer1993}. For a comprehensive treatment of this topic, we refer the reader to this book as well as to \cite{RockWets1998}.

\section{The radial distance functional} 
\label{sec:radial-distance-functional}


In this section, we establish the main properties of the radial distance functionals associated to star bodies, see Definition \ref{defn:radialdist}. As we shall see, it can be understood as an analogue, in the setting of star bodies, to the standard distance functional for closed sets.

The next remark collects the basic properties of the radial distance. For each pair $A_1,A_2\in\mathcal{S}_{rc}^{d}$, the \textit{radial excess} of $A_1$ over $A_2$ is given by
$$
e_r(A_1,A_2):=\sup\left\{
d_r(a_1,A_2):a_1\in A_1
\right\}.
$$
It is worth noticing that the radial excess is not  symmetric with respect to $A_1$ and $A_2$, and that it takes values in $[0,\infty]$.  
\begin{remark}\label{rem:properties-d-r}
The following properties hold for all $A_1,A_2\in\mathcal{S}_{rc}^{d}$:
    \begin{itemize}
        \item [(P1)]$d_r(x,A_1)=\begin{cases}
            0 &\text{ if }x\in A_1,\\
            \|x\|-\rho_{A_1}\left(
            \theta_x\right) &\text{ if }x\notin A_1.
        \end{cases}$
        \item [(P2)] $d_r(x,A_1)=0$ if and only if $x\in A_1$. 
        \item [(P3)] $A_1=A_2$ if and only if  $d_r(x,A_1)=d_r(x,A_2)$ for all $x\in\mathbb{R}^d$.
        \item [(P4)] $d(x,A_1)\leq d_r(x,A_1)$ for all $x\in\mathbb{R}^d$.
        \item [(P5)] $e(A_1,A_2)\leq e_r(A_1,A_2)$.
        \item [(P6)] $e_r(A_1,A_2)=\sup_{x\in\mathbb{R}^d}\left\{d_r(x,A_2)-d_r(x,A_1)\right\}$.
    \end{itemize}
\end{remark}
\begin{proof}
(P1) If $x\in A_1$, then $x\in A_1\tilde{+}\varepsilon B_2^d$ for every $\varepsilon>0$. Thus $d_r(x,A_1)\leq\varepsilon$, and therefore $d_r(x,A_1)=0$.
If $x\notin A_1$, then $x\neq0$ and $\rho_{A_1}\left(\theta_x\right)<\|x\|$. Moreover, 
$$
x=\rho_{A_1}\left(\theta_x\right)\theta_x+\left(\|x\|-\rho_{A_1}\left(\theta_x\right)\right)\theta_x. 
$$
Thus $d_r(x,A_1)\leq\|x\|-\rho_{A_1}\left(\theta_x\right)$. To prove the reverse inequality, let $\alpha>0$ be such that $x\in A_1\tilde{+}\alpha B_2^d$. Clearly
$\|x\|\leq\rho_{A_1\tilde{+}\alpha B_2^d}\left(\theta_x\right)=\rho_{A_1}\left(\theta_x\right)+\alpha$. Hence 
$\alpha\geq\|x\|-\rho_{A_1}\left(\theta_x\right)$, which yields to 
$d_r(x,A_1)\geq\|x\|-\rho_{A_1}\left(\theta_x\right)$.

(P2) By (P1), we only need to check the only if part. Indeed, let $x\in\mathbb{R}^d$ be such that $d_r(x;A_1)=0$. Suppose by contradiction that $x\notin A_1$. Then 
$$
0=d_r(x,A)=\|x\|-\rho_{A_1}\left(\theta_x\right),
$$ 
and $\rho_{A_1}\left(\theta_x\right)=\|x\|$. Since $A_1\in\mathcal{S}_{rc}^{d}$,  we know that $x=\rho_{A_1}\left(\theta_x\right)\theta_x\in A_1$, which is a contradiction. Therefore $x\in A_1.$

(P3) Follows from (P2).

(P4) If $x\in A_1$, then $d_r(x,A_1)=d(x,A_1)=0$. If $x\notin A_1$, then $d(x,A)\leq\|x-\rho_{A_1}\left(\theta_x\right)\theta_x\|=d_r(x,A)$.

(P5) Follows from (P4).

(P6) Let us denote by $\lambda$ the supreme $\sup_{x\in\mathbb{R}^d}\left(d_r\left(x,A_2\right)-d_r\left(x,A_1\right)\right)$. By (P1), 
$d_r(a_1,A_2)=d_r(a_1,A_2)-d_r(a_1,A_1)\leq\lambda$ for all $a_1\in A_1$. Hence,
$e_r(A_1,A_2)\leq\lambda$. 
To prove the reverse inequality, notice that for every $x\in A_1$, we have that
$$
d_r(x,A_2)-d_r(x,A_1)=d_r(x,A_2)\leq e_r(A_1,A_2).
$$ 
Similarly, for every $x'\in A_2\setminus A_1$, 
$d_r(x',A_2)-d_r(x',A_1)=-d_r(x',A_1)\leq e_r(A_1,A_2)$. Finally, by (P1), for each $x''\notin A_1\cup A_2$, we have that
$$
d_r(x'',A_2)-d_r(x'',A_1)=\rho_{A_1}(\theta_{x''})-\rho_{A_2}(\theta_{x''}).
$$
Thus, if $\rho_{A_1}(\theta_{x''})\leq\rho_{A_2}(\theta_{x''})$, trivially $d_r(x'',A_2)-d_r(x'',A_1)\leq e_r(A_1,A_2)$. Otherwise, 
$
d_r\left(\rho_{A_1}\left(\theta_{x''}\right)\theta_{x''},A_2\right)=
\rho_{A_1}(\theta_{x''})-\rho_{A_2}(\theta_{x''})
$
and 
$$
d_r(x'',A_2)-d_r(x'',A_1)= d_r(\rho_{A_1}(\theta_{x''})\theta_{x''},A_2)\leq e_r(A_1,A_2).
$$ 
From this, it follows that $\lambda\leq e_r(A_1,A_2)$ as desired.
\end{proof}



In the next proposition, we show that, as occurs with the usual distance functional associated to closed sets in $\mathbb{R}^d$, the pointwise limits
of  sequences of radial distance functionals must be radial distance functionals too.


\begin{proposition}\label{prop:completez-sec-d_r}
    Let $(A_n)_n\subseteq\mathcal{S}_{rc}^d$ be a sequence. If $(d_r(\cdot,A_n))_n$ converges pointwise to $f:\mathbb{R}^d\rightarrow\mathbb{R}$, then $f=d_r(\cdot,A)$ for some $A\in\mathcal{S}_{rc}^d$.
\end{proposition}
\begin{proof}
    First, let us notice that for every $x\in\mathbb{R}^d$ with $f(x)>0$, there is $n_x\in\mathbb{N}$ such that $d_r(x,A_n)>0$ for all $n>n_x$. Hence, by (P1)-(P2), $x\notin A_n$ and $d_r(x,A_n)=\|x\|-\rho_{A_n}(\theta_x)$. Thus $\rho_{A_n}(\theta_x)<\infty$ for every $n>n_x$, and
    $$
    f(x)=\lim_{n}\left(\|x\|-\rho_{A_n}(\theta_x)\right),
    $$
    So, $\lim_{n}\rho_{A_n}(\theta_x)$ must be finite. Now, let $\alpha:\mathbb{S}^{d-1}\rightarrow[0,\infty]$ be given by
    \begin{equation*}
        \alpha(\theta)=\begin{cases}
        \lim_{n}\rho_{A_n}(\theta)&\text{ if }f(\lambda\theta)>0\text{ for some }\lambda>0,\\
        \infty&\text{ if }f(\lambda\theta)=0\text{ for all }\lambda\geq0.      \end{cases}
    \end{equation*}
    Observe that $\alpha$ is a well defined non-negative map. Consequently, there is $A\in\mathcal{S}_{rc}^d$ such that $\rho_A=\alpha$.
    Below, we shall prove that $f(\cdot)=d_r(\cdot,A)$.
    
    Suppose that $y\notin A$. Then
    $\rho_A(\theta_y)$ is finite, and $\rho_A(\theta_y)=\alpha(\theta_y)= \lim_{n}\rho_{A_n}(\theta_y)$. Moreover, by (P1), $d_r(y,A)=\|y\|-\lim_{n}\rho_{A_n}(\theta_y)>0$. Consequently, there is $n'_y\in\mathbb{N}$ such that $\rho_{A_n}(\theta_y)<\|y\|$ for all $n>n'_y$. Hence, from (P1) again, $d_r(y,A_n)=\|y\|-\rho_{A_n}(\theta_y)$ for all $n>n'_y$, and
    $$
    f(y)=\lim_{n}d_r(y,A_n)=\lim_{n}(\|y\|-\rho_{A_n}(\theta_y))=d_r(y,A).
    $$
    On the other hand, suppose that $y\in A$, i.e.,  $d_r(y,A)=0$. Notice that $f(y)=d_r(y,A)=0$ holds for $y=0$. For $y\neq 0$, we will consider separately two cases: $\liminf_n\rho_{A_n}(\theta_y)=\infty$ and $\liminf_n\rho_{A_n}(\theta_y)<\infty$. If $\liminf_n\rho_{A_n}(\theta_y)=\infty$, then $\lim_n\rho_{A_n}(\theta_y)=\infty$ and there is $N_y\in\mathbb{N}$ such that  $y\in A_n$
    for all $n>N_y$. Hence $d_r(y,A_n)=0$ for all $n>N_y$, and  $f(y)=0$.

    If $R:=\liminf_n\rho_{A_n}(\theta_y)<\infty$, then, for every $k\in\mathbb{N}$, there is $n_k>k$ such that 
    $\rho_{A_{n_k}}(\theta_y)<R+1$. Thus $(R+2)\theta_y\notin A_{n_k}$ for all $k$, and by (P1), 
    $$
    d_r((R+2)\theta_y,A_{n_k})=R+2-\rho_{A_{n_k}}(\theta_y)>1.
    $$
    It then follows that $f((R+2)\theta_y)>1$, which implies that $\alpha(\theta_y)=\lim_n\rho_{A_n}(\theta_y)$. Moreover, since $y\in A$ and $\alpha(\theta_y)=\rho_A(\theta_y)$, then $\|y\|\leq\lim_n\rho_{A_n}(\theta_y)$.
    
    Now, observe that if the last inequality is strict, then there exist $N'_y\in\mathbb{N}$ such that  $\|y\|<\rho_{A_n}(\theta_y)$ for all $n>N'_y$. Consequently $y\in A_n$
    for all $n>N'_y$, and $f(y)=0$ as required. Otherwise, if $\|y\|=\lim_n\rho_{A_n}(\theta_y)$, then one of the following options holds: (1) there is $M_y\in\mathbb{N}$ such that $\|y\|>\rho_{A_n}(\theta_y)$ for all $n>M_y$, or  (2) there exists a subsequence $(A_{n_j})_j$ such that 
    $\|y\|\leq\rho_{A_{n_j}}(\theta_y)$ for all $j$. In the first case, $d_r(y,A_n)=\|y\|-\rho_{A_n}(\theta_y)$  for all $n>M_y$. Hence
    $$
    f(y)=\lim_{n}(\|y\|-\rho_{A_n}(\theta_y))=0.
    $$
    In the second case,
    $d_r(y,A_{n_j})=0$ for all $j$. Thus, $f(y)=\lim_{j}d_r(y,A_{n_j})=0$.
    This shows that $f(\cdot)=d_r(\cdot,A)$, as desired.
 \end{proof}
 
Below, we show that the radial distance functional provides a new way to calculate the radial metric between bounded star bodies. 

\begin{proposition}\label{prop:caract-dHr-distanciaradial}
    For every $A_1,A_2\in\mathcal{S}_{rc}^{d}$ bounded, 
    \begin{equation}\label{eq:radialHauss-radialdist}
        \delta(A_1,A_2)=
    \sup_{x\in\mathbb{R}^d}\left|d_r(x,A_1)-d_r(x,A_2)\right|
    =\max\left\{e_r(A_1,A_2),e_r(A_2,A_1)\right\}.
    \end{equation}
\end{proposition}
\begin{proof} 
    Observe that the second equality follows directly from (P6). 
    Consequently, to prove the result, it is enough to show that
    $$
    \delta(A_1,A_2)=\max\left\{e_r(A_1,A_2),e_r(A_2,A_1)\right\}.
    $$
    Let us define $\eta:=\max\left\{
    e_r(A_1,A_2),e_r(A_2,A_1)\right\}$. Note that, by (P2), $\eta=0$ if and only if $A_1=A_2$, which is equivalent to $\delta(A_1,A_2)=0$. 
    
    For the case $\eta>0$, let $\alpha>0$ be such that $A_i\subseteq A_j\tilde{+}\alpha B_2^d$ for $i,j\in\{1,2\}$ with $i\neq j$. Then, 
    for every $x\in A_1$, we have that $x\in A_2\tilde{+}\alpha B_2^d$ and $d_r(x,A_2)\leq\alpha.$ Thus $e_r(A_1,A_2)\leq\alpha$. Analogously, it follows that $e_r(A_2,A_1)\leq\alpha$. Consequently
    $\eta\leq\alpha$, which implies that $\eta\leq\delta(A_1,A_2)$.
    
    To establish the reverse inequality, observe that $d_r(x,A_2)\leq\eta$ for all $x\in A_1$. Hence, $A_1\subseteq A_2\tilde{+}\eta B_2^d$. Similarly, $A_2\subseteq A_1\tilde{+}\eta B_2^d$. Therefore $\delta(A_1,A_2)\leq\eta$, completing the proof.
\end{proof}

\subsection{Inclusions into vector spaces of real-valued functions}
\label{sec:encajes-sin-estruct}
Let us denote by $F(\mathbb{R}^d,\mathbb{R})$ the vector space of functions $f:\mathbb{R}^d\rightarrow\mathbb{R}$, and 
by $C(\mathbb{R}^d,\mathbb{R})$ the linear subspace of continuous functions.
Notice that from (P3) in Remark \ref{rem:properties-d-r}, it follows that the map 
\begin{align}\label{eq:encaje1}
     \mathcal{S}_{rc}^{d}&\longrightarrow F(\mathbb{R}^d,\mathbb{R})\\
     A&\longrightarrow d_r(\cdot,A)\nonumber
\end{align}
is injective. In particular, the identification $A\leftrightarrow d_r(\cdot,A)$ allows us to consider the family of star bodies as a subset of $F(\mathbb{R}^d,\mathbb{R})$ consisting of non-negative functions. 

It is then natural to ask for the relationship between $\mathcal{S}_{rc}^{d}$ and $C(\mathbb{R}^d,\mathbb{R})$, under the identification (\ref{eq:encaje1}). We highlight that $d_r(\cdot,A)$ is not always continuous for $A\in\mathcal{S}_{rc}^{d}$. In fact, consider a segment $[0,a]$, with $a\in\mathbb{R}^{d}\setminus\{0\}$, then
\begin{align*}
    d_r(x,[0,a])=\begin{cases}
        0,\text{ if }x\in [0,a];\\
        \|x\|-\|a\|,\text{ if } x=\lambda a\text{ and }\lambda>1;\\
        \|x\|, \text{ otherwise}.
    \end{cases}
\end{align*}
is discontinuous at any  non-zero $x\in[0,a]$. This is a notably difference with the case of closed sets in $\mathbb{R}^d$ and the 
usual distance functional. In the next proposition, we describe all star bodies $A$ for which $d_r(\cdot,A)$ is continuous.

Recall that $\mathcal{S}_{1}^{d}$ denote the family of star bodies $A\subseteq\mathbb{R}^d$ with continuous radial function $\rho_A$. We define the following subfamilies of $\mathcal{S}_{1}^{d}$:
 $\mathcal{S}_{1,b}^{d}$, consisting of bounded star bodies; $\mathcal{S}_{1,(0)}^{d}$, composed by 
  star bodies containing the origin in the interior; and $\mathcal{S}_{1, b (0)}^{d}:=\mathcal{S}_{1,(0)}^{d}\cap\mathcal{S}_{1,b}^{d}$.
  

\begin{proposition}\label{prop:car-starb-conti-distance}
Let $A\in \mathcal{S}_{rc}^{d}$. Then, the following hold
\begin{itemize} 
    \item [(1)] If $\rho_A:\mathbb{S}^{d-1}\rightarrow[0,\infty]$ is continuous, then $A$ is closed.
    \item [(2)] $d_r(\cdot,A)$ is continuous if and only if $\rho_A$ is continuous. Particularly, $d_r(\cdot,A)$ is continuous if and only if $A\in\mathcal{S}_{1}^{d}$.
\end{itemize}
\end{proposition} 
\begin{proof} (1) Let $x\in\overline{A}$. Clearly, for $x=0$, $x\in A$. Now consider $x\neq0$, and let $(a_n)_n\subseteq A\setminus\{0\}$ be a sequence converging to $x$. Note that if $\rho_{A}(\theta_{a_n})=\infty$ for all $n>n_0$, then, by continuity of $\rho_A$, $\rho_A(\theta_x)=\infty$. Hence $x\in A$. 

On the other hand, if there is a subsequence $(a_{n_k})_k$ with $\rho_{A}(\theta_{n_k})<\infty$ for all $k$, then $\|a_{n_k}\|\leq\rho_{A}\left(\theta_{a_{n_k}}\right)$. Thus, again by continuity of $\rho_{A}$, $\|x\|\leq\rho_{A}\left(\theta_x\right)$, which yields to $x\in A$. Therefore $A$ is closed.


(2) Let us suppose that $\rho_A$ is continuous and let $x\in\mathbb{R}^d$  be arbitrary. Consider a sequence $(x_n)_n\subseteq\mathbb{R}^d$ converging to $x$. Notice that
\begin{equation}\label{eq:aux-continuity-drA-radial-map2}
    d_r(x_n,A)\leq\|x_n\|\leq\|x_n-x\|+ \|x\|
\end{equation}
holds for every $n\in\mathbb{N}$ and $x$. 
Therefore, $(d_r(x_n,A))_n$ is bounded. 
We begin by proving continuity of $d_r(\cdot,A)$ when $d_r(x,A)=0$, i.e., $x\in A$. If $x=0$, then  (\ref{eq:aux-continuity-drA-radial-map2}) implies directly that $\lim_n d_r(x_n,A)=0$. 

For $x\neq 0$, we may suppose without lost of generality that there is $n_x\in\mathbb{N}$ such that $x_n\neq0$ for all $n>n_x$. 
In this case, we proceed by contradiction and suppose that $\limsup_n{d_r(x_n,A)}>0$. Then, there exist $\eta_0>0$ and a subsequence $(x_{n_k})$ such that $d_r(x_{n_k},A)>\eta_0$ for all $k$. Consequently $x_{n_k}\notin A$ and $\rho_A(\theta_{x_{n_k}})<\|x_{n_k}\|$. Since $x\in A$, $\|x\|\leq\rho_{A}(\theta_x)$. Thus, by continuity of $\rho_A$, 
$$
\|x\|\leq\rho_{A}(\theta_x)=\lim_{k\to\infty}\rho_A(\theta_{x_{n_k}})\leq\|x\|.
$$
The later in combination with (P1) proves that
$$
\eta_0\leq\lim_{k\to\infty}d_r(x_{n_k},A)=
\lim_{k\to\infty}\left\|x_{n_k}\right\|-\rho_{A}\left(\theta_{x_{n_k}}\right)
=\left\|x\right\|-\rho_{A}(\theta_{x})=0,
$$

which is a contradiction. This shows that $\limsup_n{d_r(x_n,A)}=0$, and therefore $\lim_n d_r(x_n,A)=0$ as required.

For the case $d_r(x,A)>0$, let us suppose without lost of generality that $x_n\neq0$ for all $n$.  Also notice that, by (P1)-(P2), $x\notin A$ and
$$
d_r(x,A)=\|x\|-\rho_{A}\left(\theta_x\right).
$$
By continuity of $\rho_A$, there exists $n'_x\in\mathbb{N}$ such that
$x_n\notin A$ for all $n>n'_x$. Otherwise, there is a subsequence $(x_{n_i})_{i}\subseteq A$. Particularly $\|x_{n_i}\|\leq\rho_{A}\left(\theta_{x_{n_i}}\right)$, and 
$\|x\|\leq\rho_{A}\left(\theta_{x}\right)<\|x\|$. which is a contradiction. Therefore, by (P2), 
$$
d_r(x_n,A)=\|x_n\|-\rho_{A}\left(\theta_{x_n}\right)\text{ for all }n\geq n'_x.
$$
Hence, again by continuity of $\rho_{A}$, $\lim_{n}d_r(x_n,A)=\|x\|-\rho_{A}\left(\theta_{x}\right)$. This proves that $d_r(\cdot,A)$ is continuous.

To prove the reverse implication, suppose that $d_r(\cdot,A)$ is continuous, and let $\theta\in\mathbb{S}^{d-1}$ and $(\theta_n)_n\subseteq\mathbb{S}^{d-1}$ be such that $(\theta_n)_n$ converges to $\theta$. Suppose that $\rho_A(\theta)$ is finite. By (P1), 
$$
d_r\left((\rho_A(\theta)+\varepsilon)\theta,A\right)=\varepsilon,
\text{ for every }\varepsilon>0.
$$
From the continuity of $d_r(\cdot,A)$,  it follows that $\lim_{n}d_r\left((\rho_A(\theta)+\varepsilon)
\theta_n,A\right)=\varepsilon$. Therefore, there exists $M\in\mathbb{N}$ such that $d_r\left((\rho_A(\theta)+\varepsilon)\theta_n,A\right)>0$ for all $n\geq M$. By (P1)-(P2), we have that 
$$
d_r\left((\rho_A(\theta)+\varepsilon)\theta_n,A\right)=\rho_A(\theta)+\varepsilon-\rho_A(\theta_n),
$$
and the sequence $\left(\rho_A(\theta)+\varepsilon-\rho_A(\theta_n)\right)_n$ goes to $\varepsilon$. 
Since $\varepsilon>0$ is arbitrary, then $\lim_{n}\rho_A(\theta_n)=\rho_A(\theta)$.

For the case $\rho_A(\theta)=\infty$, we shall prove that 
$\liminf\rho_A(\theta_n)=\infty$. From this, it follows that $(\rho_A(\theta_n))_n$ converges to $\infty$. Indeed,
since $\lambda\theta\in A$ for every $\lambda>0$, then
$d_r(\lambda\theta,A)=0$. Hence, by continuity of $d_r(\cdot,A)$, 
\begin{equation}\label{eq:auxprop-car-starb-conti-distance}
    \lim_{n\rightarrow\infty}d_r(\lambda\theta_n,A)=0\text{ for all } \lambda>0.
\end{equation}
We will proceed by contradiction and suppose that $\liminf\rho_A(\theta_n)<\infty$. In this case, there is a sub-sequence $(\theta_{n_i})_i$ such that $\sup_{i}\rho_A(\theta_{n_i})<R$, for some $R>0$. Then, by (P1),
$$
d_r\left(R\theta_{n_i},A\right)=R-\rho_A\left(\theta_{n_i}\right)> 
R-\sup_{i}\rho_A(\theta_{n_i})>0.
$$
Hence $\lim_{i}d_r\left(R\theta_{n_i},A\right)>0$,
which is a contradiction with equality (\ref{eq:auxprop-car-starb-conti-distance}).
This proves that $\rho_A$ is continuous. 

To finish the proof, observe that the fact that $d_r(\cdot,A)$ is continuous iff $A\in\mathcal{S}_{1}^{d}$ follows directly from what we have just proved and from (1) of this proposition.
\end{proof}

\section{A Wijsman type topology}
\label{sec:WijsmanTop}

In Section  \ref{sec:encajes-sin-estruct}, we 
established an identification of  $\mathcal{S}_{rc}^d$ as a subset of $F(\mathbb{R}^d,\mathbb{R})$, see (\ref{eq:encaje1}). Here, we further explore this relationship by introducing a weak topology on $\mathcal{S}_{rc}^d$, generated by 
the family of radial distance functionals. 

For every $x\in\mathbb{R}^d$, let $d_r(x,\cdot)$ be the map defined by 
\begin{align}\label{eq:functionalsradialWijsmantop}
    d_r(x,\cdot):\mathcal{S}_{rc}^d&\longrightarrow\mathbb{R}\\
    A&\longrightarrow d_r(x,A).\nonumber
\end{align}


\begin{definition}\label{defn:radial-WijsmanTop}
 The radial Wijsman topology $\tau_{W^r}$ on $\mathcal{S}_{rc}^d$ is the weak topology generated by the family $\left\{ d_r(x,\cdot):x\in\mathbb{R}^d\right\}$. That is, $\tau_{W^r}$ is the weakest topology on $\mathcal{S}_{rc}^d$ such that each $d_r(x,\cdot):\mathcal{S}_{rc}^d\rightarrow\mathbb{R}$ is continuous.
\end{definition}

In a similar way, we define \textit{the radial Wijsman topology on}  $\mathcal{S}_1^d$ and $\mathcal{S}_{1,b}^d$ as the weak topology on $\mathcal{S}_1^d$ (resp. $\mathcal{S}_{1,b}^d$) generated by the family 
$\{d_r(x,\cdot):\mathcal{S}_1^d\rightarrow\mathbb{R}\}_{x\in\mathbb{R}^d}$ (resp. $\mathcal{S}_{1,b}^d$).
It can be directly verified that this topology coincides with the subspace topology that  $\mathcal{S}_1^d$ and $\mathcal{S}_{1,b}^d$ inherit from 
$(\mathcal{S}_{rc}^d,\tau_{W^r})$.  Clearly, it also agrees with the topology that $\mathcal{S}_{1,b}^d$ inherits from $(\mathcal{S}_{1}^d,\tau_{W^r})$.


It is well-known that the convergence in a weak topology amounts to the pointwise convergence of the defining family of functions. Particularly, for the radial Wijsman topology, a net $(A_i)_{i\in I}\subset\mathcal{S}_{rc}^d$ converges to some $A\in\mathcal{S}_{rc}^d$ if and only if $(d_r(\cdot,A_i)_i)_{i\in I}$ converges pointwise to $d_r(\cdot,A)$. That is,
\begin{equation}\label{eq:pointwise-convergence-rdistance}
 \tau_{W^r}\text{-}\lim_{i}A_i=A\textit{ iff }\lim_{i}d_r(x,A_i)=d_r(x,A),\textit{ for every }x\in\mathbb{R}^d.
\end{equation}

Notice that, under the identification  (\ref{eq:encaje1}), 
the radial Wijsman topology  is the topology that $\mathcal{S}_{rc}^d$ inherits from $F(\mathbb{R}^d,\mathbb{R})$, when the latter is equipped with the topology of pointwise convergence (see e.g., \cite[Definition VII.1]{Nagatabook}).

Below, we establish a characterization of convergence in the radial Wijsman topology, formulated in terms of the pointwise convergence of the associated radial functions. As we will see throughout this work, this characterization plays a crucial role in simplifying proofs related to $\tau_{W^r}$.




 
For $A\in\mathcal{S}_{rc}^d$ (resp. $\mathcal{S}_{1}^d$ and $\mathcal{S}_{1,b}^d$), 
let us write $\rho(\theta,A):=\rho_{A}(\theta)$ and consider the family of radial functions
$\left\{\rho(\theta,\cdot):\mathcal{S}_{rc}^d\rightarrow[0,\infty]\right\}_{\theta\in\mathbb{S}^{d-1}}$ (resp. $\mathcal{S}_{1}^d$ and $\mathcal{S}_{1,b}^d$). The weak topology  on $\mathcal{S}_{rc}^d$ (resp. $\mathcal{S}_{1}^d$ and $\mathcal{S}_{1,b}^d$) determined by this family of maps is called \textit{the topology of pointwise convergence of radial functions}. The next theorem shows that this topology agrees with $\tau_{W^r}$ on  $\mathcal{S}_{rc}^d$.

\begin{theorem}\label{thm:caract-rWijsmanTopology}
    Let $(A_i)_{i\in I}\subset\mathcal{S}_{rc}^d$ be a net and  $A\in\mathcal{S}_{rc}^d$. Then, $(A_i)_{i\in I}$ is $\tau_{W^r}$-convergent  to $A$ if and only if  $(\rho_{A_i}(\cdot))_i$ converges pointwise to $\rho_{A}(\cdot)$. The same holds for $\mathcal{S}_{1}^d$. 
\end{theorem}

\begin{proof}
    First let us suppose that $(\rho_{A_i}(\cdot))_i$ converges pointwise to $\rho_{A}(\cdot)$. We shall prove that $d_r(.,A_i)_{i\in I}$ converges pointwise to $d_r(.,A)$, see (\ref{eq:pointwise-convergence-rdistance}).  Notice that for $x=0$, 
    $d_r(0,A)=d_r(0,A_i)=0$ for all $i\in I$. Thus, the convergence follows directly. 
    
    For $x\in A\setminus\{0\}$.  Observe that $d_r(x,A)=0$ and that
    $\rho_{A}(\theta_x)\geq\|x\|>0$.  Hence,
    $$
    \lim_{i}\rho_{A_i}(\theta_x)=\rho_{A}(\theta_x)\geq\|x\|.
    $$
    If $\rho_{A}(\theta_x)>\|x\|$, there is $i_0\in I$ such that $\rho_{A_i}(\theta_x)>\|x\|$ for all $i>i_0$. Thus,  $x\in A_i$  and $d_r(x,A)=d_r(x,A_i)=0$ for all $i>i_0$. This proves that $(d_r(x,A_i))_i$ converges to $d_r(x,A)$. 
    
    

    On the other side, if $\rho_{A}(\theta_x)=\|x\|$, then for every $0<\varepsilon<\|x\|$, there is $i_1\in I$ such that $|\rho_{A_i}(\theta_x)-\|x\||<\varepsilon$ for all $i>i_1$. Consequently, for $i>i_1$ such that $\rho_{A_i}(\theta_x)\geq\|x\|$, $d_r(x,A_i)=0$. Otherwise, for $i>i_1$  such that
    $\rho_{A_i}(\theta_x)<\|x\|$, we have that $d_r(x,A_i)=\|x\|-\rho_{A_i}(\theta_x)<\varepsilon$ (see (P1)). Therefore $d_r(x,A_i)<\varepsilon$ for all $i>i_1$, and the net  $(d_r(x,A_i))_i$ goes to $0$.


    For the case $x\notin A$. Notice that, by (P1), 
    $d_r(x,A)=\|x\|-\rho_{A}(\theta_x)>0$. Hence,  for each $0<\delta<d_r(x,A)$, there is $i_4\in I$ such that $|\rho_{A_i}(\theta_x)-\rho_{A}(\theta_x)|<\delta$ for every $i>i_4$. Thus $\rho_{A_i}(\theta_x)<\|x\|$ and, by  (P1), $d_r(x,A_i)=\|x\|-\rho_{A_i}(\theta_x)$. From this, it follows that for every $i>i_4$,
    $$
    |d_r(x,A)-d_r(x,A_i)|=|\rho_{A_i}(\theta_x)-\rho_{A}(\theta_x)|<\delta.
    $$  
    Therefore $(d_r(x,A_i))_i$ converges to $d_r(x,A)$, and the first implication is proved.

    To prove the reverse implication, suppose that $(A_i)_{i\in I}$ $\tau_{W^r}$-converges  to $A$ and let $\theta\in\mathbb{S}^{d-1}$. For $\rho_{A}(\theta)$ finite, let $\varepsilon>0$ be arbitrary. In this case, $(\rho_{A}(\theta)+\frac{\varepsilon}{2})\theta\notin A$ and,  by (P1),
    $$
    d_r\left(\left(\rho_{A}(\theta)+\frac{\varepsilon}{2}\right)\theta,A\right)=\frac{\varepsilon}{2}.
    $$
    Hence, by our hypothesis and (\ref{eq:pointwise-convergence-rdistance}), there exists $i'\in I$ such that 
    $$
    \left|d_r\left(\left(\rho_{A}(\theta)+\frac{\varepsilon}{2}\right)\theta,A_i\right)-\frac{\varepsilon}{2}\right|<\frac{\varepsilon}{2}
    $$
    for all $i>i'$.  Thus, $\left(\rho_{A}(\theta)+\frac{\varepsilon}{2}\right)\theta\notin A_i$. Moreover, from (P1) and the last inequality, it follows that $|\rho_{A}(\theta)-\rho_{A_i}(\theta)|<\varepsilon$
    for all $i>i'$. Therefore $(\rho_{A_i}(\theta))_i$ converges to $\rho_{A}(\theta)$.

    For $\rho_{A}(\theta)=\infty$, we shall prove that 
    $\liminf\rho_{A_i}(\theta)=\infty$. From this, it follows directly that $(\rho_{A_i}(\theta))_i$ goes to $\infty$. In fact, since $M\theta\in A$ for all $M>0$, then $d_r(M\theta,A)=0$. Hence, by  (\ref{eq:pointwise-convergence-rdistance}), 
    \begin{equation}\label{eq:aux-caract-rWijsman}
      \lim_{i\in I}d_r(M\theta,A_i)=0\text{ for all } M>0.
    \end{equation}
    We will proceed by contradiction and suppose that $\liminf\rho_{A_i}(\theta)<\infty$. In this case, there is $R>0$ such that 
    $\inf_{j\geq i} \rho_{A_j}(\theta)<R$ for all $i\in I$. Therefore, for every $i\in I$, there is $j\geq i$ such that $\rho_{A_j}(\theta)<R$. Thus 
    $(R+1)\theta\notin A_j$ and, by (P1),
    $$
    d_r\left((R+1)\theta,A_{j}\right)=R+1-\rho_{A_{j}}\left(\theta\right)>1
    $$
    which leads to a contradiction with (\ref{eq:aux-caract-rWijsman}) when $M=R+1$. This proves that  $(\rho_{A_i}(\cdot))_i$ is pointwise convergent to $\rho_{A}(\cdot)$.

    To finish the proof, notice that the results just proved remain valid with $\mathcal{S}_{1}^d$ instead of $\mathcal{S}_{rc}^d$. Therefore, the topologies also coincide on $\mathcal{S}_{1}^d$.
\end{proof}


The following corollary gathers basic properties of the radial Wijsman topology and its connection with the radial metric $\delta$. Notably, $\tau_{W^r}$ is not a metrizable topology, even when restricted to the family $\mathcal{S}_{1,b}^d$. 

Let us denote by $C_p\left(\mathbb{S}^{d-1}\right)$ the space of continuous functions $f:\mathbb{S}^{d-1}\rightarrow\mathbb{R}$, equipped with the topology of pointwise convergence. Similarly, let $C_p^+(\mathbb{S}^{d-1})
\subseteq C_p\left(\mathbb{S}^{d-1}\right)$ denote the cone  of non-negative functions.


\begin{corollary}\label{cor:rWijsman-nonmetrizable} The following hold:
    \begin{enumerate}

         \item The space $(\mathcal{S}_{1,b}^d,\tau_{W^r})$ is homeomorphic to $C_p^+(\mathbb{S}^{d-1})$.  Particularly, 
         $\tau_{W^r}$ is not a metrizable topology on $\mathcal{S}_{1,b}^d$, $\mathcal{S}_{1}^d$ and $\mathcal{S}_{rc}^d$.

        \item On $\mathcal{S}_{1,b}^d$ and $\mathcal{S}_{rc,b}^d$, $\tau_{W^r}\subseteq\tau_{\delta}$ and the inclusion is strict. 

    \end{enumerate}
    
\end{corollary}
\begin{proof}
    
    (1) From Theorem \ref{thm:caract-rWijsmanTopology}, it follows that the map
    \begin{align}\label{eq:homeomorfismo-S1b-Continuas-Esfera}
     (\mathcal{S}_{1,b}^d,\tau_{W^r})&\longrightarrow C_p^+\left(\mathbb{S}^{d-1}\right)\\
     A&\longrightarrow \rho_A(\cdot)\nonumber
    \end{align}
    is a homeomorphism. However, $C_p^+\left(\mathbb{S}^{d-1}\right)$ is not metrizable (see e.g. \cite[Theorem 9.5]{Dugundji66} ), which implies that $(\mathcal{S}_{1,b}^d,\tau_{W^r})$ is also not metrizable. 
    Consequently, since $(\mathcal{S}_{1,b}^d,\tau_{W^r})$ is a subspace of both $(\mathcal{S}_{1}^d,\tau_{W^r})$ and $(\mathcal{S}_{rc}^d,\tau_{W^r})$, these spaces are not metrizable either.

    (2) From Theorem \ref{thm:caract-rWijsmanTopology}, it follows that $\tau_{W^r}\subseteq\tau_{\delta}$. 
    To prove that the topologies are different, consider the following example: 
    For each $n\in\mathbb{N}$, let $E_n$ be the star body with 
    \begin{equation}\label{examp:rWijsmanNotWijsman-converg}
        \rho_{E_n}(\theta)=\frac{n|\langle\theta,e_1\rangle|}{e^{n|\langle\theta,e_1\rangle|}}.
    \end{equation}
    Clearly $ \rho_{E_n}(\theta)\leq1$ for all $\theta\in\mathbb{S}^{d-1}$. Hence, by Proposition \ref{prop:car-starb-conti-distance}-(1), each  $E_n\in\mathcal{S}_{1,b}^d$. Since $(\rho_{E_n}(\cdot))$ is pointwise convergent to $\rho_{\{0\}}(\cdot)\equiv0$, then by Theorem \ref{thm:caract-rWijsmanTopology},
    $(E_n)_n$ $\tau_{W^r}$-converges to $\{0\}\in\mathcal{S}_{1,b}^d$. However  $(\rho_{E_n}(\cdot))$ is not uniformly convergent to $\rho_{\{0\}}(\cdot)$, and therefore $(E_n)_n$ does not converges to $\{0\}$ with respect to the radial metric.
\end{proof}
\begin{remark}
    $\mathcal{S}_{1}^d$ is not closed in $(\mathcal{S}_{rc}^d,\tau_{W^r})$.
\end{remark}
\begin{proof}
Below, we exhibit a $\tau_{W^r}$-convergent sequence of star bodies in $\mathcal{S}_{1}^d$ whose limit does not belong to this class. Indeed, let $X_n$, $n\in\mathbb{N}$, be the star body with radial function
$$
\rho_{X_n}(\theta)=|\langle\theta,e_1\rangle|^n.
$$ 
By Proposition \ref{prop:car-starb-conti-distance},  $X_n\in\mathcal{S}_{1}^d$. However, $(\rho_{X_n})_n$ converges pointwise to $\rho_{[-e_1,e_1]}$.
Thus, by Theorem \ref{thm:caract-rWijsmanTopology}, $(X_n)_n$ is $\tau_{W^r}$-convergent to $[-e_1,e_1]\notin\mathcal{S}_{1}^d$. 
\end{proof}


\subsection{On the relation with the Wijsman topology}
\label{sec:rel-Wr-W}
Recall that \textit{the Wijsman topology $\tau_W$} on the family of non-empty closed sets of $\mathbb{R}^d$, $CL(d)$,  is the weak topology generated by the family of distance functionals  $\left\{d(x,\cdot):x\in\mathbb{R}^d\right\}$, where $d(x,\cdot)$ is given by
\begin{align*}
    d(x,\cdot):CL(d)&\rightarrow[0,\infty)\\
    C&\rightarrow d(x,C).
\end{align*}
We refer to \cite{Beer1993} for the basic properties of $\tau_W$. 
An important  consequence of Corollary \ref{cor:rWijsman-nonmetrizable} regards with the relation between the Wijsman and the radial Wijsman topologies. In this sense, recall that $\tau_{W}$ is metrizable on $CL(d)$ (\cite[Theorem 2.1.5]{Beer1993}), while, from (1) in Corollary \ref{cor:rWijsman-nonmetrizable}, it follows that  $\tau_{W^r}$ is not metrizable on the family of closed star bodies $\mathcal{S}_{1}^d$. As a consequence, the topologies $\tau_{W}$ and $\tau_{W^r}$ do not agree on $\mathcal{S}_{1}^d$. Below, we shall see that on $\mathcal{S}_{1}^d$ none of these topologies includes the other one.


In Example \ref{exam:ejemploMoszynska}, we exhibited a sequence $(A_n)_n$ of star bodies in $\mathcal{S}_{1,b}^d$ that converges to $B_2^d$ in the Hausdorff metric (and thus in $\tau_{W}$), but such that its associated sequence of radial functions does not converge pointwise to $\rho_{B_2^d}$. Consequently, from Theorem \ref{thm:caract-rWijsmanTopology}, we know that $(A_n)_n$ does not converge  to $B_2^d$ in $\tau_{W^r}$. Therefore, $\tau_{W^r}\nsubseteq\tau_{W}$ The next example shows that $\tau_{W}\nsubseteq\tau_{W^r}$ on $\mathcal{S}_{1,b}^d$.

\begin{example}
    Let $(E_n)_n\subset\mathcal{S}_{1,b}^d$ be the sequence in (\ref{examp:rWijsmanNotWijsman-converg}). We already know that $(E_n)_n$ is $\tau_{W^r}$-convergent to $\{0\}$.  We shall see that it does not converge to $\{0\}$ in the Wijsman topology. Consequently on $\mathcal{S}_{1,b}^d$, $\tau_{W}\nsubseteq\tau_{W^r}$. Suppose that $(E_n)_n$ $\tau_W$-converges to $\{0\}$, and let $(\theta_n)_n\subset\mathbb{S}^{d-1}$ be the sequence given by
    \begin{equation}\label{eq:examp-limclosednotclosed}
        \theta_n=\left(\frac{1}{n},\sqrt{1-\frac{1}{n^2}}\right).
    \end{equation}
    Clearly, $(\theta_n)_n$ goes to $e_2$ and $\rho_{E_n}(\theta_n)=e^{-1}$ for all $n$. Hence, $e^{-1}\theta_n\in E_n$ and 
    $d(e^{-1}e_2,E_n)\leq\|e^{-1}e_2-e^{-1}\theta_n\|$. Thus, 
    $$
    d(e^{-1}e_2,\{0\})=\lim_{n\rightarrow\infty}d(e^{-1}e_2,E_n)=0
    $$
    which is a contradiction. Therefore, $(E_n)_n$ does not converge to $\{0\}$ in $\tau_W$.
\end{example}

\cite[Theorem 3.1]{Wijsman1966} provides a characterization of convergence in the Wijsman topology. It shows that a sequence $(Y_n)_n\subseteq CL(d)$ converges, in $\tau_{W}$,  to a closed set $Y\subseteq\mathbb{R}^d$ if and only if $\lim_{n}d(y,Y_n)=0$ for all $y\in Y$, and $\liminf_{n}d(x,X_n)>0$ for all $x\notin Y$. This theorem, in combination  with (P4), proves the following:

\begin{remark}
    Let $(A_n)_n\subset\mathcal{S}_{rc}^d$ be a sequence of closed sets with  $\tau_{W^r}$-$\lim A_n=A$ for some $A\in\mathcal{S}_{rc}^d$ closed. Then $A=\tau_{W}$-$\lim A_n$  if and only if $\liminf d(x,A_n)>0$ for all $x\notin A$.
\end{remark}


We finish this section with some results regarding the relation between the radial Wijsman topology and the Wijsman topology on $\mathcal{K}_{0}^d$. In this sense, it is convenient to notice that
$\tau_{W^r}$-limits of closed star bodies are not always closed, not even when only compact convex sets are considered. To check this, let $(\theta_n)_n\subseteq\mathbb{R}^d$ be the sequence of (\ref{eq:examp-limclosednotclosed}), and denote by $H_n$ the closed half-space $\left\{z\in\mathbb{R}^d:\langle z,\theta_n\rangle\leq0\right\}$. It follows from Theorem \ref{thm:caract-rWijsmanTopology} that the sequence
$(H_n)_n$ is $\tau_{W^r}$-convergent to
$\left\{z\in\mathbb{R}^d:\langle z,e_2\rangle<0\right\}\cup\left\{\lambda e_1:\lambda\leq0\right\}$, which is a non-closed convex set. Despite this fact, convexity is indeed preserved by $\tau_{W^r}$-limits of closed convex sets in $\mathcal{K}_{0}^d$.

\begin{proposition}\label{prop:Wrlimits-convex-isconvex}
    Let $(K_i)_{i\in I}\subset\mathcal{K}_{0}^d$ be a net and $A\in\mathcal{S}_{rc}^d$. If $A=\tau_{W^r}$-$\lim_{i}K_i$, then $A$ is convex. Particularly, if $A$ is  closed, $A\in\mathcal{K}_{0}^d$.
\end{proposition}
\begin{proof}
    Let $a,b\in A$ and $t\in[0,1]$. Clearly, for $a=0$ or $b=0$, $[a,b]\subseteq A$. Suppose that $a,b\neq0$. Since $A\in\mathcal{S}_{rc}^d$, then $[a,b]\subseteq A$ whenever $a$ and
    $b$ are linearly dependent.  For linearly independent $a$ and $b$, we shall prove that
    $(1-\varepsilon)(ta+(1-t)b)\in A$ for every  $0<\varepsilon<1$. Consequently,
    $$
    \rho_A\left(\frac{ta+(1-t)b}{\|ta+(1-t)b\|}\right)\geq(1-\varepsilon)
    \|ta+(1-t)b\|.
    $$
    Since $\varepsilon$ is arbitrary and $A$ is radially closed, we must have that $ta+(1-t)b\in A$. 
    To show that $(1-\varepsilon)(ta+(1-t)b)\in A$, observe that $\|x-\varepsilon x\|<\rho_A(\theta_x)$ for $x=a,b$. Now, since  $A=\tau_{W^r}$-$\lim_{i}K_i$, then, by Theorem \ref{thm:caract-rWijsmanTopology},
    there exist $i_a,i_b\in I$ such that $\|x-\varepsilon x\|<\rho_{K_i}(\theta_x)$ for $i>i_x$ with $x=a,b$. Hence, there is $i'>i_a,i_b$ such that $a-\varepsilon a, b-\varepsilon b\in K_i$ for all $i>i'$. It then follows that $(1-\varepsilon)(ta+(1-t)b)\in K_i$, and
    $$
    d_r((1-\varepsilon)(ta+(1-t)b),A)=\lim_{i}d_r((1-\varepsilon)(ta+(1-t)b),K_i)=0.
    $$
    Therefore, by (P2), $(1-\varepsilon)(ta+(1-t)b)\in A$ as desired.
\end{proof}

In the next proposition, we establish conditions under which $\tau_{W^r}$-convergent sequences of convex sets also convergence in $\tau_{W}$. To achieve this, we provide criteria for convergence in the Attouch-Wets metric $d_{AW}$ (\ref{eq:dAW}). Since $\tau_{W}$ and $d_{AW}$ are equivalent on $\mathcal{K}_{0}^d$ (see e.g. \cite[Theorem 3.1.4]{Beer1993}) our desired result follows directly.  

\begin{proposition}\label{prop:conditionsWr-implies-dAW}
    Let $(K_n)_n\subseteq\mathcal{K}_{0}^d$ and $K\in\mathcal{K}_{0}^d$ be such that $K=\tau_{W^r}$-$\lim K_n$. Then, the following hold:
    \begin{enumerate}
        \item There is a subsequence $(K_{n_m})_m$ that converges in $d_{AW}$ to some $K'\in\mathcal{K}_{0}^d$, with $K\subseteq K'$. Moreover, this inclusion holds for the limit of any $d_{AW}$-convergent subsequence of $(K_n)_n$.
        \item If $\eta B_2^d\subseteq K_n, K$ for some $\eta>0$  and all $n\in\mathbb{N}$, then  $K=d_{AW}$-$\lim K_n$.
    \end{enumerate}
\end{proposition}
\begin{proof}
    (1)  In \cite[Proposition 3.3]{NataliaLuisa2}, it is proved that $(\mathcal{K}_{0}^d,d_{AW})$ is compact. Consequently, there exist a subsequence $(K_{n_m})_m$ and $K'\in\mathcal{K}_{0}^d$ such that 
    $K'=d_{AW}$-$\lim_{m} K_{n_m}$. In this case, by (P4), for every $x\in K$, we have that
    $$
    d(x,K')\leq d_r(x,K')=\lim_{m}d_r(x,K_{n_m})=d_r(x,K)=0
    $$
    Thus, $x\in K'$ and $K\subseteq K'$. It is straightforward that this inclusion holds for every $A\in\mathcal{K}_{0}^d$ that is the $d_{AW}$-limit of a subsequence from $(K_n)_n$.

    (2) Consider $x,y\in\mathbb{R}^d\setminus\{0\}$, and let $0\leq\psi_{x,y}\leq\pi$ be the angle between the vectors 
    $x$ and $y$. Then $\cos(\psi_{x,y})=\frac{\langle x,y\rangle}{\|x\|\|y\|}$.
    We shall prove that every subsequence of $(K_n)_n$ has a $d_{AW}$-convergent sub-subsequence whose limit is $K$. From this, it follows that the whole sequence converges to $K$.
    
    By (1) of this proposition, every subsequence  $(K_{n_m})_m$ has a $d_{AW}$-convergent sub-subsequence $(K_{n_{m_j}})_j$ with limit $A\in\mathcal{K}_{0}^d$ such that $K\subseteq A$.
    To prove that $A\subseteq K$,  consider an arbitrary $a\in A\setminus\{0\}$.  Since $A=d_{AW}$-$\lim K_{n_{m_j}}$, there exist a sequence $(x_j)_j\subseteq\mathbb{R}^d$ such that
    $x_j\in K_{n_{m_j}}$ and  $d(a,K_{n_{m_j}})=\|a-x_j\|$, for every $j$. Thus $(x_j)_j$  converges to $a$, and we can suppose  without lost of generality that $\|x_j\|, \|a\|>\eta$ for all $j$. For each $j\in\mathbb{N}$, let $Sec_j$ be the section of the ball $\eta B_2^d$ given by 
    $$
    Sec_j:=\left\{\eta u:u\in\mathbb{S}^{d-1}\text{ and }\langle u,x_j\rangle=\eta\right\},
    $$
    and let $C_j:=\cco(\{x_j\}\cup \eta B_2^d)$. By our hypothesis, $C_j\subseteq K_{n_{m_j}}$ for each $j$, i.e., $\rho_{C_j}(\cdot)\leq\rho_{K_{n_{m_j}}}(\cdot)$. 
    
    First, let us consider the case that the angle $\psi_{x_j,a}=0$ for infinitely many $j\in\mathbb{N}$. In this case, we have a subsequence $(x_{j_i})_i$ such that $a=\lambda_ix_{j_i}$ for some $\lambda_i>0$. Hence, for all $i\in\mathbb{N}$, $\theta_a=\theta_{x_{j_i}}$ and
    $$
    \|x_{j_i}\|=\rho_{C_{j_i}}(\theta_a)\leq \rho_{K_{n_{m_{j_i}}}}(\theta_a).
    $$ 
    Since $K=\tau_{W^r}$-$\lim K_n$, then, by Theorem \ref{thm:caract-rWijsmanTopology},
    $\|a\|\leq\rho_K(\theta_a)$.
    Therefore $a\in K$ as desired. 
    
    Now, suppose that $\psi_{x_j,a}>0$ for all $j>N_0$ and some positive integer $N_0$. Since $(x_j)_j$ approaches to $a$, and $\|x_j\|, \|a\|>\eta$ for all $j$, then there is an integer $N_1>N_0$ such that for every $j>N_1$ and every $\eta u\in Sec_j$ with $u\in\mathbb{S}^{d-1}$,
    \begin{equation}\label{eq:Aux2-convrW-impliesAW}
    \frac{\langle u,x_j\rangle}{\|x_j\|}=\frac{\eta}{\|x_j\|}<\frac{1}{2}+\frac{\eta}{2\|a\|}<\frac{\langle a,x_j\rangle}{\|a\|\|x_j\|}.
    \end{equation}
    It then follows that $\psi_{a,x_j}<\psi_{u,x_j}$ for all $j>N_1$. On the other hand, there also exists an integer $N_2>N_1$ such that, for all $j>N_2$ and every $\eta u\in Sec_j$ with $u\in\mathbb{S}^{d-1}$,
    \begin{equation}\label{eq:Aux3-convrW-impliesAW}
        \langle a,u\rangle>0.
    \end{equation}
    Indeed, since  $(x_j)_j$ goes to $a$, we can choose $N_2$ so that 
    $\|a-x_j\|<\frac{\eta}{2}$ for all $j>N_2$. This yields to 
    $|\langle a-x_j,u\rangle|\leq\|a-x_j\|<\frac{\eta}{2}$ from which (\ref{eq:Aux3-convrW-impliesAW}) follows.

    Now, for every $\eta u\in Sec_j$ with $u\in\mathbb{S}^{d-1}$, let
    $H_j(u)$ be the half-space $\{z\in\mathbb{R}^d:\langle z,x_j-\eta u\rangle>0\}$. Clearly $x_j\in H_j(u)$ for all $j$. We claim that
    there is an integer $N_3>N_2$ such that for every $j>N_3$ and every $\eta u\in Sec_j$, $a\in H_j(u)$. 
    Otherwise for each $i\in\mathbb{N}$, there is $\eta u_{j_i}\in Sec_{j_i}$ with $u\in\mathbb{S}^{d-1}$ such that $\langle a,x_{j_i}-\eta u_{j_i}\rangle\leq0$. Hence,
    $$
    \langle a,x_{j_i}\rangle\leq\langle a,\eta u_{j_i}\rangle\leq\|a\|\eta<\|a\|^2
    $$
    which is a contradiction, since $(x_j)_j$ approaches to $a$. 
    
    Hereafter, for every $j>N_3$, let
    $E_j$ denote the 2-dimensional subspace generated by $a$ and $x_j$.
    Notice that $E_j\cap Sec_j\neq\emptyset$ for every $j>N_3$. Moreover, for each $\eta u\in E_j\cap Sec_j$, with $u\in\mathbb{S}^{d-1}$, we have that $a,x_j\in E_j\cap H_j(u)$ and $\psi_{a,x_j}<\psi_{u,x_j}$ (see (\ref{eq:Aux2-convrW-impliesAW})). Therefore, 
    for every $j>N_3$, there is $\eta u_j\in E_j\cap Sec_j$, with $u_j\in\mathbb{S}^{d-1}$, such that $a$ belongs to the relative interior of the cone generated by the origin and the vectors $x_j$ and $u_j$. Since $j>N_3$ and the vectors $a$ and $x_j$ are linearly independent, then by (\ref{eq:Aux3-convrW-impliesAW}), there exist $\lambda_j>0$ and $t_j\in(0,1)$ such that 
    $$
    \lambda_ja=t_jx_j+(1-t_j)\eta u_j\in C_{j}\text{ for all }j>N_3.
    $$
    By compactness of $\mathbb{S}^{d-1}$ and $[0,1]$, there are subsequences $(u_{j_i})_i$ and $(t_{j_i})_i$ converging to $u_0\in\mathbb{S}^{d-1}$ and $t_0\in[0,1]$, respectively. Thus, the sequence $(\lambda_{j_i})_i$ must converge to some $\lambda_0\geq0$, and we have that $$
    \lambda_0a=t_0a+(1-t_0)\eta u_0.
    $$ 
    Since $(u_{j_i})_i$ and $(x_{j_i})_i$ converge to $u_0$ and $a$, respectively, it follows that $\langle a,u_0\rangle=\eta<\|a\|$. As a result, $a$ and $u_0$ are linearly independent, which implies that $\lambda_0=t_0=1$. Finally, from the fact that $C_{j_i}\subseteq K_{n_{m_{j_i}}}$,  we have that
    $$
    \lambda_{j_i}\|a\|\leq\rho_{C_{j_i}}(\theta_a)\leq\rho_{K_{n_{m_{j_i}}}}(\theta_a) \text{ for all }i.
    $$
    Hence, from Theorem \ref{thm:caract-rWijsmanTopology}, it follows that $\rho_K(\theta_a)\geq\lambda_0\|a\|=\|a\|$, which implies that $a\in K$. This proves that $A\subseteq K$, as desired.
\end{proof}

The next example shows that the inclusion in Proposition \ref{prop:conditionsWr-implies-dAW}-(1) might be strict even if the sequence of convex sets belongs to $\mathcal{S}_{1,b}^d$. 
\begin{example}\label{exam:truncated-parabola}
    For every $n\in\mathbb{N}$, let $P_n$ be the truncated parabola
    $$
        P_n:=\left\{(x,y)\in\mathbb{R}^2: -1\leq x\leq 1\text{ and }0\leq y\leq\frac{x^2}{n}   \right\}.
    $$
    Then $(P_n)_n\subseteq\mathcal{S}_{1,b}^2\cap\mathcal{K}_0^2$ and 
    $\{0\}=\tau_{W^r}$-$\lim P_n$. However, 
    $$
    [-1,1]\times\{0\}=d_{AW}\text{-}\lim P_n.
    $$
    It is straightforward to check that for every $n\in\mathbb{N}$,  
    $P_n\subseteq\left[-1,1\right]\times\left[0,\frac{1}{n}\right]$ and 
    $[-1,1]\times\{0\}\subseteq P_n+\frac{1}{n}B_2^2$. 
    Therefore, $(P_n)_n$ converges to $[-1,1]\times\{0\}$ with respect to $d_H$. Since both metrics $d_H$ and $d_{AW}$ generate the same topology on $\mathcal{K}_{0,b}^d$ \cite[Theorem 3.2]{SakaiYaguchi2006}, the first part is proved. On the other hand, the $\tau_{W^r}$-convergence of $(P_n)_n$ to $\{0\}$ follows directly from the pointwise convergence of the radial maps and Theorem \ref{thm:caract-rWijsmanTopology}.
\end{example}

\section{An Attouch-Wets type distance}
\label{sec:radial Attouch-Wets topology}

Let $F_{ucb}(\mathbb{R}^d,\mathbb{R})$ denote the space of maps $f:\mathbb{R}^d\rightarrow\mathbb{R}$ equipped 
with the topology $\tau_{ucb}$ of the uniform convergence on bounded sets of $\mathbb{R}^d$. See for instance \cite[Example VII.3]{Nagatabook}).

\begin{definition}\label{def:radialAW}
    The radial Attouch-Wets topology  $\tau_{AW^r}$ on $\mathcal{S}_{rc}^d$ is the topology that $\mathcal{S}_{rc}^d$ inherits from $F_{ucb}(\mathbb{R}^d,\mathbb{R})$, under the identification given by $A\rightarrow d_r(\cdot,A).$ 
\end{definition}

We can make use of Proposition \ref{prop:car-starb-conti-distance} to define the \textit{radial Attouch-Wets topology $\tau_{AW^r}$ on $\mathcal{S}_{1}^d$}. In this case, it is the topology that $\mathcal{S}_{1}^d$ inherits from $C(\mathbb{R}^d,\mathbb{R})$, endowed with $\tau_{ucb}$, under the usual identification $A\rightarrow d_r(\cdot,A)$.
It is not difficult to prove that this topology coincides with the subspace topology that $\mathcal{S}_{1}^d$ obtain from $(\mathcal{S}_{rc}^d,\tau_{AW^r})$.

As in the case of the Hausdorff, Wijsman and Attouch-Wets topologies on $CL(d)$, there is a basic relation between the metric topology $\delta$, and the topologies $\tau_{W^r}$ and $\tau_{AW^r}$ on the family of star bodies:

\begin{remark}\label{rem:dAWR-stronger-than-Wr}
   $\tau_{W^r}\subseteq\tau_{AW^r}$ on $\mathcal{S}_{rc}^d$. Additionally, on $\mathcal{S}_{rc,b}^d$,  $\tau_{W^r}\subseteq\tau_{AW^r}\subseteq\tau_{\delta}$. 
\end{remark}
The previous remark follows directly from the definition of the topologies.

Below, we show that the map $d_{AW^r}$, defined in (\ref{defn:rd-AW}), is a metric on $\mathcal{S}_{rc}^d$ that generates the radial Attouch-Wets topology on $\mathcal{S}_{rc}^d$.
\begin{theorem}\label{thm:rAW-complete-metrizable} 
$d_{AW^r}$
is a metric on $\mathcal{S}_{rc}^d$ that is compatible with $\tau_{AW^r}$. Furthermore, $(\mathcal{S}_{rc}^d,d_{AW^r})$ is a complete metric space. 
\end{theorem}
\begin{proof}
    Let $A_1, A_2, A_3\in \mathcal{S}_{rc}^d$  and $j\in\mathbb{N}$. The fact that $d_{AW^r}$ is a metric follows from (P3) and the following inequality 
    \begin{align*}
    \min\left\{\frac{1}{j},\sup_{\|x\|\leq j}\left|d_{r}(x,A_1)-d_r(x,A_2)\right|\right\}\leq\min\left\{\frac{1}{j},\sup_{\|x\|\leq j}\left|d_{r}(x,A_1)-d_r(x,A_3)\right|\right\}\\+
    \min\left\{\frac{1}{j},\sup_{\|x\|\leq j}\left|d_{r}(x,A_3)-d_r(x,A_2)\right|\right\},
    \end{align*}
    which is a consequence of the triangle inequality. 
    
    To prove that $d_{AW^r}$ and $\tau_{AW^r}$ generate the same topology, we will check that they have the same convergent nets. Let
    $(Y_i)_{i\in I}\subseteq\mathcal{S}_{rc}^d$ be a net converging in $\tau_{AW^r}$ to $Y\in\mathcal{S}_{rc}^d$. We shall see that $(Y_i)_i$ is $d_{AW^r}$-convergent to $Y$. Indeed, let $\eta>0$ and pick $j_0\in\mathbb{N}$ such that $\eta\in\left(\frac{1}{j_0+1},\frac{1}{j_0}\right]$. Since $(d_r(\cdot,Y_i))_i$ converges uniformly on $j_0B_2^d$ to $d_r(\cdot,Y)$, then there is $i_0\in I$ such that for all $i\geq i_0$, 
    $$
    \sup_{x\in j_0B_2^d}\left|d_{r}(x,Y)-d_r(x,Y_i)\right|<
    \eta\leq\frac{1}{j_0}.
    $$
    Hence, 
    $$
    \sup_{j\leq j_0}\min\left\{\frac{1}{j}, 
    \sup_{\|x\|\leq j}\left|d_{r}(x,A_1)-d_r(x,A_2)\right|\right\}<\eta.
    $$ 
    This inequality and the fact that $\eta>\frac{1}{j_0+1}$ imply that $d_{AW^r}(Y,Y_i)<\eta$ for all $i\geq i_0$. Therefore,
    $Y=d_{AW^r}\text{-}\lim Y_i$ as desired. 
    
    On the other hand, consider a sequence $(X_n)_n\subseteq\mathcal{S}_{rc}^d$ converging in $d_{AW^r}$ to $X\in\mathcal{S}_{rc}^d$. Below we prove that $(X_n)_n$ is $\tau_{AW^r}$-convergent to $X$. Let $C\subseteq\mathbb{R}^d$ be a bounded set, and let $R_0\in\mathbb{N}$ be such that $\overline{C}\subseteq R_{0}B_2^d$. Then, for every $0<\varepsilon<\frac{1}{R_0+1}$, there is $n_0\in\mathbb{N}$ such that
    $d_{AW^r}(X,X_n)<\varepsilon$ for all $n>n_0$. Hence 
    $$
    \sup_{x\in C}\left|d_{r}(x,X)-d_r(x,X_n)\right|\leq\sup_{\|x\|\leq R_0+1}\left|d_{r}(x,X)-d_r(x,X_n)\right|<\varepsilon.
    $$
    Therefore $(d_r(\cdot,X_n))_n$ converges uniformly on $C$ to $d_r(\cdot,X)$. Since $C$ is arbitrary, $(X_n)_n$ is $\tau_{AW^r}$-convergent to $X$.

    To prove that $(\mathcal{S}_{rc}^d,d_{AW^r})$ is complete, let $(A_n)_n\subseteq\mathcal{S}_{rc}^d$ be a Cauchy sequence. Then, for every $x\in\mathbb{R}^d$, $(d_r(x,A_n))_n$ is also a Cauchy sequence and we can define $f:\mathbb{R}^d\rightarrow[0,\infty)$ as 
    $$
    f(x)=\lim_{n}d_r(x,A_n).
    $$
    From Proposition \ref{prop:completez-sec-d_r}, $f=d_r(\cdot, A)$ for some $A\in\mathcal{S}_{rc}^d$. We shall show that $(A_n)_n$ is  $d_{AW^r}$-convergent to $A$.  In fact, let $\gamma>0$ be such that $\gamma\in\left(\frac{1}{j+1},\frac{1}{j}\right]$  for some $j\in\mathbb{N}$. Since $(A_n)_n$ is a Cauchy sequence, there is $M\in\mathbb{N}$ such that $d_{AW^r}(A_n,A_m)<\frac{\gamma}{2}$ for all $n,m>M$. Particularly, for all $n,m>M$,
    $$
    \sup_{\|x\|\leq j}\left|d_{r}(x,A_n)-d_r(x,A_m)\right|<\frac{\gamma}{2}.
    $$
    Thus, if we fix $n>M$ and consider an arbitrary $x\in jB_2^d$, we have that
    $$
    |d_r(x,A)-d_r(x,A_n)|=\lim_{m}|d_r(x,A_m)-d_r(x,A_n)|\leq\frac{\gamma}{2}.
    $$
    Hence, $\sup_{\|x\|\leq j}\left|d_{r}(x,A_n)-d_r(x,A_m)\right|<\gamma$ and $d_{AW^r}(A,A_n)<\gamma$ for all $n>M$. This proves that 
    $A$ is the $d_{AW^r}$-limit of $(A_n)_n$ as desired.
\end{proof}

The next lemma can be understood as a generalization of \cite[Lemma 2.1]{NataliaLuisa2} to the context of the radial Attouch-Wets distance. It significantly eases the treatment of this distance and will be used frequently through the work.  

\begin{lemma}\label{lemm:axu-calcular-rdAW}
    Let $A_1, A_2\in\mathcal{S}_{rc}^d$ and $\eta>0$. Then the following statements hold:
    \begin{enumerate}
        \item $d_r(x,A_1)=d_r(x,A_1\cap\eta B_2^d)$ for all $x\in \eta B_2^d$.

        \item $\delta(A_1\cap\eta B_2^d,A_2\cap\eta B_2^d)=\sup_{\|x\|\leq \eta}\left|d_{r}(x,A_1)-d_r(x,A_2)\right|$.

        \item $d_{AW^r}(A_1,A_2)=\sup_{j\in\mathbb{N}}\min\left\{\frac{1}{j}, 
        \delta(A_1\cap jB_2^d,A_2\cap jB_2^d)\right\}$.

        \item For every integer $j\geq 1$ and $\varepsilon\in\left(\frac{1}{j+1},\frac{1}{j}\right]$, $d_{AW^r}(A_1, A_2)<\varepsilon$ if and only if $\delta(A_1\cap jB_2^d,A_2\cap jB_2^d)<\varepsilon$.
    \end{enumerate}
\end{lemma}
\begin{proof}
    (1) The result holds easily for $x\in A_1$. For 
    $x\in\eta B_2^d\setminus A_1$, we have that
    \begin{align*}
        d_{r}(x,A_1)=\|x\|-\rho_{A_1}(\theta_x)&=\|x\|-\min\{\rho_{A_1}(\theta_x),\eta\}\\
        &=\|x\|-\rho_{A_1\cap\eta B_2^d}(\theta_x)\\
        &=d_r(x,A_1\cap\eta B_2^d).
    \end{align*}

    (2) From (1) of this proposition and Proposition \ref{prop:caract-dHr-distanciaradial}, it follows that
    \begin{align*}
    \delta(A_1\cap\eta B_2^d,A_2\cap\eta B_2^d)&=\sup_{x\in\mathbb{R}^d}|d_r(x,A_1\cap\eta B_2^d)-d_r(x,A_2\cap\eta B_2^d)|\\
    &\geq\sup_{\|x\|\leq \eta}\left|d_{r}(x,A_1\cap\eta B_2^d)-d_r(x,A_2\cap\eta B_2^d)\right|\\
    &=\sup_{\|x\|\leq \eta}\left|d_{r}(x,A_1)-d_r(x,A_2)\right|.
    \end{align*}
    To show the reverse inequality, notice that if $x_0\in A_1\cap\eta B_2^d$, then 
    \begin{align*}
     d_r(x_0,A_2\cap\eta B_2^d)&=|d_r(x_0,A_1\cap\eta B_2^d)-d_r(x_0,A_2\cap\eta B_2^d)|\\
     &\leq\sup_{\|x\|\leq \eta}\left|d_{r}(x,A_1\cap\eta B_2^d)-d_r(x,A_2\cap\eta B_2^d)\right|\\
    &=\sup_{\|x\|\leq \eta}\left|d_{r}(x,A_1)-d_r(x,A_2)\right|.
    \end{align*}
    Similarly, for every $y_0\in A_2\cap\eta B_2^d$, we have that 
    $$
    d_r(y_0,A_1\cap\eta B_2^d)\leq\sup_{\|x\|\leq \eta}\left|d_{r}(x,A_1)-d_r(x,A_2)\right|.
    $$
    Finally, by using Proposition \ref{prop:caract-dHr-distanciaradial} again, we get that
    \begin{align*}
    \delta(A_1\cap\eta B_2^d,A_2\cap\eta B_2^d)&=\max\{e_r(A_1\cap\eta B_2^d,A_2\cap\eta B_2^d),e_r(A_2\cap\eta B_2^d,A_1\cap\eta B_2^d)\}\\
    &\leq\sup_{\|x\|\leq \eta}\left|d_{r}(x,A_1)-d_r(x,A_2)\right|.
    \end{align*}

    (3) Follows from (2) of this proposition.

    (4) Follows from (2) and (3) of this Proposition, and the definition of $d_{AW^r}$.
\end{proof}

\begin{proposition} \label{prop:subspaces-rdAW}The following statements hold:
    \begin{enumerate}
        \item $\mathcal{S}_1^d$ is a closed subset of $(\mathcal{S}_{rc}^d,d_{AW^r})$. Particularly, $(\mathcal{S}_{1}^d,d_{AW^r})$ is complete.
        
        \item $\mathcal{S}_{rc,b}^d$ is an open subset of $(\mathcal{S}_{rc}^d,d_{AW^r})$.

        \item $d_{AW^r}$ and $\delta$ generate the same topology on $\mathcal{S}_{rc,b}^d$.
    \end{enumerate}
\end{proposition}
\begin{proof}

(1)   
Let $A\in\mathcal{S}_{rc}^d$ be such that $A=d_{AW^r}\text{-}\lim A_n$, for some $(A_n)_n\subseteq\mathcal{S}_1^d$. By Proposition \ref{prop:car-starb-conti-distance}, $(d_{r}(\cdot,A_n))_n\subseteq C(\mathbb{R}^d,\mathbb{R})$. Since $(d_{r}(\cdot,A_n))_n$ converges uniformly on compact sets to  $d_r(\cdot,A)$,  $d_r(\cdot,A)$ must be a continuous function. Therefore, by Proposition \ref{prop:car-starb-conti-distance}, $A\in\mathcal{S}_{1}^d$. The completeness of $(\mathcal{S}_{1}^d,d_{AW^r})$ follows from the fact that it is a closed subset of a complete metric space (Theorem \ref{thm:rAW-complete-metrizable}).

(2) Let $A\in\mathcal{S}_{rc,b}^d$ and let $j_0\geq1$ be an integer such that $A\subseteq j_0B_2^d$. We shall see that if $0<\varepsilon<\frac{1}{2j_0}$, then, for every $X\in\mathcal{S}_{rc}^d$ with $d_{AW^r}(A,X)<\varepsilon$, 
\begin{equation}\label{eq:Src-bouded}
    X\subseteq2j_0B_2^d.
\end{equation}
Suppose that there is $x\in X$ such that $\|x\|>2j_0$. Then $2j_0\theta_{x}\in X\cap 2j_0B_2^d$. This, in combination with Lemma \ref{lemm:axu-calcular-rdAW}-(1)-(3), yields to
$$
d_r(2j_0\theta_{x},A)=d_r(2j_0\theta_{x},A\cap2j_0B_2^d)\leq\delta(A\cap2j_0B_2^d,X\cap2j_0B_2^d)<\varepsilon<\frac{1}{2j_0}.
$$
Thus, there is $a\in A$ such that $0$, $2j_0\theta_x$ and $a$ are collinear and $\|a-2j_0\theta_{x}\|<\frac{1}{2j_0}$. This yields to $\|a\|>j_0$ which is a contradiction, since $A\subseteq j_0B_2^d$. 

(3) We shall show that the identity 
$I:(\mathcal{S}_{rc,b}^d,\delta)\rightarrow(\mathcal{S}_{rc,b}^d,d_{AW^r})$  is a homeomorphism. From Proposition \ref{prop:caract-dHr-distanciaradial}, $d_{AW^r}(A_1,A_2)\leq\delta(A_1,A_2)$ for all $A_1,A_2\in\mathcal{S}_{rc,b}^d$. Hence, $I$
is continuous. To prove that it is a homeomorphism, let $A\in\mathcal{S}_{rc,b}^d$  be such that 
$A\subseteq j_0B_2^d$ for some $j_0\geq1$ integer, and consider $0<\varepsilon<\frac{1}{2j_0}$. From (\ref{eq:Src-bouded}), it follows that $X\subseteq2j_0B_2^d$, for every $X\in\mathcal{S}_{rc}^d$ with $d_{AW^r}(A,X)<\varepsilon$. Thus, from Lemma \ref{lemm:axu-calcular-rdAW}-(3), if $d_{AW^r}(A,X)<\varepsilon$, then
$$
\delta(A,X)=\delta(X\cap2j_0B_2^d,A\cap2j_0B_2^d)<\varepsilon.
$$
Therefore $I$ is a homeomorphism.
\end{proof}

\subsection{On the relation with the  Attouch-Wets topology}
\label{sec:rel-AWr-AW}
Recall that the \textit{Attouch-Wets topology $\tau_{AW}$} on $CL(d)$ is the topology that $CL(d)$ inherits from the space of continuous functions $C(\mathbb{R}^d,\mathbb{R})$, equipped with the topology of the uniform convergence on bounded sets of $\mathbb{R}^d$, under the identification $C\rightarrow d(\cdot,C)$ given by the distance functional (see e.g., \cite[Definition 3.1.2]{Beer1993}).

It is well-known that $\tau_{AW}$ is a metrizable topology. In fact, the following distance, which is called \textit{the Attouch-Wets distance $d_{AW}$}, is an admissible distance for $\left(CL(d),\tau_{AW}\right)$. For $C_1,C_2\in CL(d)$, we define
\begin{equation}
\label{eq:dAW}
d_{AW}(C_1,C_2):=\sup_{j\in\mathbb{N}}\left\{\min\left\{\frac{1}{j},\sup_{\|x\|<j}|d(x,C_1)-d(x,C_2)|\right\}\right\}.
\end{equation}
This definition is the one used in \cite{SakaiYaguchi2006}. However it is equivalent to the Attouch-Wets distance defined in \cite[Definition 3.1.2]{Beer1993}. We refer \cite{Beer1993} for a thorough treatment of $\tau_{AW}$.

Below we exhibit the relation between $d_{AW}$ and $d_{AW^r}$ on the families of closed star bodies and of closed convex sets $\mathcal{K}_0^d$. We begin by showing 
that \cite[Lemma 2.1]{NataliaLuisa2} holds not only for $\mathcal{K}_0^d$ but for all closed star bodies. This result allows to prove that convergence in the radial Attouch-Wets topology implies convergence in the Attouch-Wets topology (see Proposition \ref{prop:rdAW-es-top-strnger-than-rdAW}).


\begin{lemma}\label{lem:extension-resultado-Natalia}
    Let $\eta>0$ and let $A\in\mathcal{S}_{rc}^d$ be closed. Then, for all $x\in\eta B_2^d$,
    \begin{equation}\label{eq:Aux-extension-lema-estrellados}
          d(x,A)=d(x,A\cap\eta B_2^d). 
    \end{equation}
    Particularly, for every $A_1,A_2\in\mathcal{S}_{rc}^d$, the following hold:
    \begin{enumerate}
         \item $d_{AW}(A_1,A_2)=\sup_{j\in\mathbb{N}}\min\left\{\frac{1}{j}, 
        d_H(A_1\cap jB_2^d,A_2\cap jB_2^d)\right\}$.

        \item For every integer $j\geq 1$ and $\varepsilon\in\left(\frac{1}{j+1},\frac{1}{j}\right]$, $d_{AW}(A_1, A_2)<\varepsilon$ if and only if $d_{H}(A_1\cap jB_2^d,A_2\cap jB_2^d)<\varepsilon$.
    \end{enumerate}
\end{lemma}
\begin{proof}
    Let $x\in\eta B_2^d$ and consider the set $W=\left\{a\in A: d(x,A)=\|x-a\|\right\}$. Notice that $W\neq\emptyset$, since $A$ is closed. We
    shall show that $W\subseteq\eta B_2^d$. From this, it follows that
    $d(x,A)\leq d(x,A\cap\eta B_2^d)\leq\|x-a\|=d(x,A)$ for any $a\in W$, which proves the equality.
    
    We proceed by contradiction. Suppose that there is $a\in W$ such that $\|a\|>\eta$. Let $p$ denote the orthogonal projection $\frac{\langle x,a\rangle}{\|a\|^2}a$. Observe that $\|p\|\leq \eta$ and that $\frac{|\langle x,a\rangle|}{\|a\|^2}<1$. Thus, 
    $p\in(-a,a)\cap\eta B_2^d$. If $p\in(-a,0]\cap\eta B_2^d$, then $\|p\|<\|p-a\|$ and
    \begin{equation*}
           \|x\|^2=\|x-p\|^2+\|p\|^2<\|x-p\|^2+\|p-a\|^2=\|x-a\|^2=d(x,A)^2
    \leq\|x\|^2
    \end{equation*}
    which is a contradiction. On the other hand, if $p\in[0,a)\cap\eta B_2^d$, then $p\in A$ because $A\in\mathcal{S}_{rc}^d$. Moreover, 
    let us denote by $\mathbb{R}^{+}a$ the ray $\left\{ta: t\geq0\right\}$. Then, since $\mathbb{R}^{+}a$ is a closed convex set and $p\neq a$, we must have that
    $$
    d(x,\mathbb{R}_{+}a)=\|x-p\|<\|x-a\|.
    $$ 
    However, $\|x-a\|=d(x,A)\leq\|x-p\|$ which leads to a contradiction with the previous inequality. This yields to $W\subseteq\eta B_2^d$ as desired.

    To finish the proof, notice that (1) and (2) follow directly from the equality in  (\ref{eq:Aux-extension-lema-estrellados}) and the definition of the Hausdorff and the Attouch-Wets metrics.
\end{proof}

The following relation between $d_{AW}$ and $d_{AW^r}$ can be understood as an extension of (\ref{eq:d-H-menor-rd-H}) to the context of unbounded star bodies.

\begin{proposition}\label{prop:rdAW-es-top-strnger-than-rdAW}
    For every pair $A_1,A_2\subseteq\mathbb{R}^d$ of closed star bodies, inequality in (\ref{eq:dAW-es-menor-rdAW}) holds.
\end{proposition}
\begin{proof}
    Follows from inequality in (\ref{eq:d-H-menor-rd-H}), Lemma \ref{lemm:axu-calcular-rdAW}-(3) and Lemma \ref{lem:extension-resultado-Natalia}-(1).
\end{proof}


\begin{proposition} The following statements hold:
\begin{enumerate}
    \item If $(A_n)_n\subseteq\mathbb{R}^d$ is a sequence of closed star bodies converging in $d_{AW^r}$ to $A\in\mathcal{S}_{rc}^d$, then $\overline{A}=d_{AW}\text{-}\lim_{n}A_n$.
    
    \item On $\mathcal{S}_{1}^d$, $\tau_{AW}\subsetneq\tau_{AW^r}$.
\end{enumerate}
\end{proposition}
\begin{proof}
    (1) By Proposition \ref{prop:rdAW-es-top-strnger-than-rdAW}, $(A_n)_n$ is a Cauchy sequence in $(CL(d),d_{AW})$. Since the later is a complete metric space (see e.g., \cite[Theorem 3.1.3]{Beer1993}), there is $Z\in CL(d)$ such that $Z=d_{AW}\text{-}\lim_{n}A_n$. We shall see that $Z=\overline{A}$. Indeed, by (P4), for every $a\in A$ we have that
    $$
    d(a,Z)=\lim_{n}d(a,A_n)\leq\lim_{n}d_r(a,A_n)=d_r(a,A)=0.
    $$
    Therefore $A\subseteq Z$, and $\overline{A}\subseteq Z$. To show the reverse inclusion, let $z\in Z$ and consider a sequence $(x_n)_n\subseteq\mathbb{R}^d$ such that $x_n\in A_n$ and $d(z,A_n)=\|z-x_n\|$ for all $n$. Clearly, $(x_n)_n$ converges to $z$. Moreover, there are $j_0, n_0\in\mathbb{N}$ such that $\|z\|<j_0$ and $\|x_n\|<j_0$ for all $n>n_0$. Now, consider $0<\varepsilon<\frac{1}{j_0}$. Then, there is $n_1>n_0$ such that $d_{AW^r}(A,A_n)<\varepsilon$. Hence for all $n>n_1$,
    $$
    d_r(x_n,A)=|d_r(x_n,A)-d_r(x_n,A_n)|\leq\sup_{\|x\|\leq j_0}\left|d_{r}(x,A)-d_r(x,A_n)\right|<\varepsilon,
    $$
    which proves that $\lim_{n}d_r(x_n,A)=0$. By (P4), for every $n$, we have that
    $$
    d(x_n,\overline{A})=d(x_n,A)\leq d_r(x_n,A),
    $$ 
    which yields to $d(z,\overline{A})=0$. Therefore $z\in\overline{A}$, and thus $Z\subseteq\overline{A}$ as desired.
  
    (2) The inclusion follows directly from inequality (\ref{eq:dAW-es-menor-rdAW}). To check that
    $\tau_{AW}\neq\tau_{AW^r}$, recall that Example \ref{exam:ejemploMoszynska} exhibits a sequence in $\mathcal{S}_{1,b}^d$ 
    converging in $d_H$ (and thus in $d_{AW}$) to $B_2^d$, but such that it is not $\delta$-convergent to $B_2^d$. Thus, by
    Proposition \ref{prop:subspaces-rdAW}-(3), this sequence is not $d_{AW^r}$-convergent to $B_2^d$ either. This proves the result.
\end{proof}

In the next remark we detail some basic properties of the radial Attouch-Wets topology on $\mathcal{K}_0^d$.

\begin{remark}\label{rem:d_AWr de convexos}
\begin{enumerate}
    \item On $\mathcal{K}_0^d$, $\tau_{AW}\subsetneq\tau_{AW^r}$. 

    \item On $\mathcal{K}_{0,b}^d$, $\delta$ and $d_{AW^r}$ generate the same topology. 
    \item The metric $d_H$, $d_{AW}$, $\delta$, and $d_{AW^r}$ generate the same topology on $\mathcal{K}_{(0),b}^d$.
\end{enumerate}
    
\end{remark}
\begin{proof}
    (1) By inequality (\ref{eq:dAW-es-menor-rdAW}), $\tau_{AW}\subseteq\tau_{AW^r}$. To prove that the inclusion is strict, let $(\theta_{n})_n\subseteq\mathbb{S}^{d-1}\setminus\{e_1\}$ be a sequence converging to $e_1$. Then, it is not difficult to check that the sequence of segments $\left([0,\theta_n]\right)_n$   converges in $d_H$ (and therefore in $d_{AW}$) to $[0,e_1]$. However, for all $n$,
    $$
    \delta\left([0,e_1],[0,\theta_n]\right)\geq\left|\rho_{[0,e_1]}(\theta_n)-\rho_{[0,\theta_n]}(\theta_n)\right|=1.
    $$
    Hence, $\left([0,\theta_n]\right)_n$ does not converge in $\delta$ to $[0,e_1]$. By Proposition \ref{prop:subspaces-rdAW}-(3),  $\left([0,\theta_n]\right)_n$ does not converge in $d_{AW^r}$ to $[0,e_1]$ either.  
    
    (2) Follows from Proposition \ref{prop:subspaces-rdAW}-(3).

    (3) Follows from Proposition \ref{prop:subspaces-rdAW}-(3) and the well-known fact that on $\mathcal{K}_{(0),b}^d$, the metrics $\delta$ and $d_{H}$ generate the same topology.
\end{proof}

\section{Dualities}
\label{sec:dualities}


A \textit{duality} for the family $\mathcal{S}_{rc}^d$ (resp. $\mathcal{S}_{1}^d$, $\mathcal{S}_{1,b}^d$ or $\mathcal{S}_{1,(0),b}^d$) is a map $\mathcal{T}:\mathcal{S}_{rc}^d\rightarrow\mathcal{S}_{rc}^d$ (resp. $\mathcal{S}_{1}^d$, $\mathcal{S}_{1,b}^d$ or $\mathcal{S}_{1,(0),b}^d$) with the following properties:
\begin{itemize}
    \item [(D1)] $\mathcal{T}\circ\mathcal{T}$ is the identity map.
    \item [(D2)] For every $A_1,A_2\in\mathcal{S}_{rc}^d$, if $A_1\subseteq A_2$, then
    $\mathcal{T}(A_2)\subseteq\mathcal{T}(A_1)$.
\end{itemize}
This definition as well as basic results and recent developments regarding dualities can be consulted in \cite{ArtsteinMilman2007,ArtsteinMilman2008,ZooDualities,Slomka2011}.


The star duality $\Phi$, defined in (\ref{eq:duality-Src}), was introduced in \cite[Definition 9.11]{MMilmanRotem}. However, as it is pointed out there, this duality was originally studied in \cite{Moszynska1999} for the  family of compact star bodies $\mathcal{S}_{1,(0),b}^d$. Indeed,  \cite[Definition 3.2]{Moszynska1999} introduces the duality $\phi:\mathcal{S}_{1,(0),b}^d\rightarrow\mathcal{S}_{1,(0),b}^d$ defined, for every $A\in\mathcal{S}_{1,(0),b}^d$, as
\begin{equation}\label{eq:duality-Moszynska}
\phi(A)=\overline{\mathbb{R}^d\setminus i(A)}. 
\end{equation}
Here $i(A)$ is the image of $A$ under the spherical inversion $i:\mathbb{R}^d\setminus\{0\}\rightarrow\mathbb{R}^d\setminus\{0\}$, given by $i(x)=\frac{x}{\|x\|^2}$. 
As highlighted in \cite{MMilmanRotem}, both of the dualities $\Phi$ and $\phi$ agree on $\mathcal{S}_{1,(0),b}^d$. That is,
\begin{equation}\label{eq:dualitiesMosz-and-Milman}
\Phi(A)=\overline{\mathbb{R}^d\setminus i(A)}\text{ for every }A\in\mathcal{S}_{1,(0),b}^d,
\end{equation}
see \cite{Moszynska1999} or \cite[$\mathsection$15.4]{MoszynskaBook}. 

It is worth noticing that equality (\ref{eq:dualitiesMosz-and-Milman}) does not hold on $\mathcal{S}_{rc}^d$. In fact, for every $u\in\mathbb{S}^{d-1}$,  $\Phi([0,u])=\mathbb{R}^{d}\setminus\{tu:t>1\}$ and $\overline{\mathbb{R}^d\setminus i([0,u])}=\mathbb{R}^d$. 
The next proposition exhibits the relation between $\Phi$ and the expression in the right side of (\ref{eq:dualitiesMosz-and-Milman}) on the family $\mathcal{S}_{rc}^d$.

\begin{proposition}\label{prop:duality-S1}
    For every $A\in\mathcal{S}_{rc}^d$, 
    $\Phi(A)=rc(\mathbb{R}^d\setminus i(A)).$ 
    Particularly, if $X\in\mathcal{S}_{1}^{d}$, then $\Phi(X)\in\mathcal{S}_{1}^{d}$ and
    $$
    \Phi(X)=\overline{\mathbb{R}^d\setminus i(X)}\textit{ for every }X\in\mathcal{S}_{1}^{d}.
    $$ 
\end{proposition}
\begin{proof}
Let us notice that for every $A\in\mathcal{S}_{rc}^d$, $\mathbb{R}^d\setminus i(A)$ is a star. Indeed, $0\in\mathbb{R}^d\setminus i(A)$ and, for each  non-zero $x\in\mathbb{R}^d\setminus i(A)$, we have that $\rho_{A}(\theta_x)<\frac{1}{\|x\|}$. Thus, if $\lambda x\notin\mathbb{R}^d\setminus i(A)$  for some $\lambda\in(0,1)$, then $\lambda x\in i(A)$ and $\rho_{A}(\theta_x)\geq\frac{1}{\lambda\|x\|}$ which is a contradiction. Hence $\lambda x\in\mathbb{R}^d\setminus i(A)$ for every $\lambda\in[0,1]$.

By  the previous observation, in order to prove that $\Phi(A)=rc(\mathbb{R}^d\setminus i(A))$, it is enough to show that 
\begin{equation}\label{eq:radial-map-of-a-duality}
    \rho_{\mathbb{R}^d\setminus i(A)}=\frac{1}{\rho_A}\textit{ for every }A\in\mathcal{S}_{rc}^d.
\end{equation}

Let $\theta\in\mathbb{S}^{d-1}$ be such that $\rho_A(\theta)\in(0,\infty)$. Then
$\frac{1}{\rho_A(\theta)}\theta\in i(A)$ and 
$\rho_{\mathbb{R}^d\setminus i(A)}(\theta)\leq\frac{1}{\rho_A(\theta)}$. To prove the reverse inequality, consider an arbitrary $\alpha>\rho_A(\theta)$. Clearly, $\alpha\theta\notin A$ and
$\frac{1}{\alpha}\theta\notin i(A)$. Thus $\rho_{\mathbb{R}^d\setminus i(A)}(\theta)\geq\frac{1}{\alpha}$, and $\rho_{\mathbb{R}^d\setminus i(A)}(\theta)\geq\frac{1}{\rho_A(\theta)}$.  

For $\theta\in\mathbb{S}^{d-1}$ with $\rho_A(\theta)=0$. Let $\alpha>0$ be arbitrary and notice that $\alpha\theta\notin A$ iff
$\frac{1}{\alpha}\theta\notin i(A)$ iff $\rho_{\mathbb{R}^d\setminus i(A)}(\theta)\geq\frac{1}{\alpha}$ iff $\rho_{\mathbb{R}^d\setminus i(A)}(\theta)=\infty$.

For $\theta\in\mathbb{S}^{d-1}$ such that $\rho_A(\theta)=\infty$, we have that $t\theta\in A$ for all $t>0$. Hence, $\frac{1}{t}\theta\in i(A)$ and 
$\rho_{\mathbb{R}^d\setminus i(A)}(\theta)\leq\frac{1}{t}$. Therefore $\rho_{\mathbb{R}^d\setminus i(A)}(\theta)=0$. 
This proves (\ref{eq:radial-map-of-a-duality}). 

To finish the proof, observe that the map $\frac{1}{\rho_X}:\mathbb{S}^{d-1}\rightarrow[0,\infty]$ is continuous for every $X\in\mathcal{S}_{1}^{d}$. Hence, by Proposition \ref{prop:car-starb-conti-distance}-(1), $\Phi(X)$ is closed. Consequently, $\Phi(X)\in\mathcal{S}_{1}^{d}$ and   $\Phi(X)=\overline{\mathbb{R}^d\setminus i(X)}$. 
\end{proof}

Below, we examine the continuity of the star duality with respect to the radial metric, the Wijsman topology and the radial Attouch-Wets topology, respectively.

\begin{remark}\label{rem:dualidad-continua-acotados}
    The duality map $\Phi:\mathcal{S}_{1,(0),b}^d\rightarrow\mathcal{S}_{1,(0),b}^d$ is continuous with respect to the radial metric. 
\end{remark}
\begin{proof}
    Consider any $A\in\mathcal{S}_{1,(0),b}^d$. Let $\varepsilon>0$ and set $r_0=\min_{\theta\in\mathbb{S}^{d-1}}|\rho_A(\theta)|>0$. Then, for every $X\in\mathcal{S}_{1,(0),b}^d$ with $\delta(A,X)<\min\left\{\frac{r_0}{2},\frac{r_{0}^2\varepsilon}{2}\right\}$, we have that
    $$
    \delta\left(\Phi(A),\Phi(X)\right)\leq
    \frac{2\delta\left(A,X\right)}{r_0^2}<\varepsilon.
    $$
    This proves the continuity of $\Phi$.
\end{proof}

\begin{remark} The duality $\Phi:\mathcal{S}_{1,(0),b}^d\rightarrow\mathcal{S}_{1,(0),b}^d$ is discontinuous with respect to the Hausdorff metric $d_H$. 
\end{remark}

To verify the remark, notice that the sequence $(A_n)_n\subseteq\mathcal{S}_{1, (0),b}^d$ from Example \ref{exam:ejemploMoszynska} converges in $d_H$ to $B_2^d$, however 
$d_H\left(\Phi(A_n),B_2^d\right)\geq3$ for every $n$.

\begin{theorem}\label{thm:duality-Wr-continuous}
    The duality $\Phi:\mathcal{S}_{rc}^d\rightarrow\mathcal{S}_{rc}^d$ is continuous with respect to $\tau_{W^r}.$
\end{theorem}
\begin{proof} 
    Let $A\in\mathcal{S}_{rc}^d$. We shall prove that for every net $(A_i)_{i\in I}\subseteq\mathcal{S}_{rc}^d$ that is $\tau_{W^r}$-convergent to $A$, the net $(\Phi(A_i))_{i\in I}$ converges to $\Phi(A)$ in $\tau_{W^r}$.  By Theorem \ref{thm:caract-rWijsmanTopology}, the convergence in $\tau_{W^r}$ amounts to the pointwise convergence of the radial maps $(\rho_{A_i})_{i}$ to $(\rho_{\Phi(A_i)})_{i}$. Consequently, we only need to prove that for every $\theta\in\mathbb{S}^{d-1}$, the net $\left(\rho_{\Phi(A_i)}(\theta)\right)_{i}$ converges to $\rho_{\Phi(A)}(\theta)$. 
    
    We consider separately the cases when $\rho_{A}(\theta)$ strictly positive, $0$ and $\infty$. For the first case, notice that $\rho_{A}(\theta)\in(0,\infty)$ if and only if $\rho_{\Phi(A)}(\theta)=\frac{1}{\rho_{A}(\theta)}\in(0,\infty)$. Since $(\rho_{A_i}(\theta))_{i}$ converges to $\rho_{A}(\theta)$, there is $i_0\in I$ such that $\rho_{A_i}(\theta)\in(0,\infty)$ for all $i>i_0$. Then, it is a straightforward matter to prove that $\left(\rho_{\Phi(A_i)}(\theta)\right)_i$
    approaches to $\rho_{\Phi(A)}(\theta)$. For the case $\rho_{A}(\theta)=0$ (resp. $\rho_{A}(\theta)=\infty$), observe that 
    $\rho_{\Phi(A)}(\theta)=\infty$ 
    (resp. $\rho_{\Phi(A)}(\theta)=0$). It then follows that the net $\left(\rho_{\Phi(A_i)}(\theta)\right)_i$ converges to $\infty$ (resp. $0$) when $(\rho_{A_i}(\theta))_{i}$ approaches to $0$ (resp. to $\infty$).
\end{proof}

\begin{lemma}\label{lem:Src(0)-is-open}
     Let $A\in\mathcal{S}_{rc}^d$ be such that $\frac{1}{j_0}B_2^d\subseteq A$ for some integer $j_0\geq1$, and let $X\in\mathcal{S}_{rc}^d$ be such that  $d_{AW^r}(X,A)<\frac{1}{2j_0}$. Then $\frac{1}{2j_0}B_2^d\subseteq X$. Particularly, $\mathcal{S}_{rc,(0)}^d$ is an open set in $\left(\mathcal{S}_{rc}^d,d_{AW^r}\right)$.  
\end{lemma}
\begin{proof}
    Since $A\in\mathcal{S}_{rc}^d$ is such that $\inf_{\theta\in\mathbb{S}^{d-1}}\rho_A(\theta)\geq\frac{1}{j_0}$ and $d_{AW^r}(A,X)<\frac{1}{2j_0}$, then,
    by Lemma \ref{lemm:axu-calcular-rdAW}-(3), we have that for every $\theta\in\mathbb{S}^{d-1}$, 
    $$
    \left|\min{\{\rho_X(\theta),2j_0\}}-\min\{\rho_A(\theta),2j_0\}\right|<\frac{1}{2j_0}.
    $$
    Hence,  $\min{\{\rho_X(\theta),2j_0\}}>\frac{1}{2j_0}$ and
    $\rho_X(\theta)>\frac{1}{2j_0}$ for all $\theta\in\mathbb{S}^{d-1}$. Therefore, $\frac{1}{2j_0}B_2^d\subseteq X$ as required.
\end{proof}

\begin{theorem}\label{thm:duality-awr-continua}
    The duality $\Phi:\mathcal{S}_{rc}^d\rightarrow\mathcal{S}_{rc}^d$ is continuous with respect to $d_{AW^r}$.
\end{theorem}
\begin{proof}
    We will consider three separate cases: (1) $A\in\mathcal{S}_{rc}^d$ is bounded; (2) $A\in\mathcal{S}_{rc}^d$ contains the origin in the interior; and (3) $A\in\mathcal{S}_{rc}^d$ satisfies $\inf_{\theta\in\mathbb{S}^{d-1}}\rho_{A}(\theta)=0$ and
    $\sup_{\theta\in\mathbb{S}^{d-1}}\rho_{A}(\theta)=\infty$.
    
    We begin by proving the continuity of $\Phi$ for the third case. The remaining cases will follow from minor modifications in the proof of this case. 
    
    Let $\varepsilon\in\left(\frac{1}{j+1},\frac{1}{j}\right]$ with $j\geq2$ an integer, and 
    consider  $A\in\mathcal{S}_{rc}^d$ as in (3) of this proof.  We shall prove that if
    $$
    0<\eta<\min\left\{\frac{1}{2j+3},\frac{\varepsilon}{j^2+j\varepsilon} \right\},
    $$ 
    then, for every $X\in\mathcal{S}_{rc}^d$ with 
    $d_{AW^r}(A,X)<\eta$, 
    $d_{AW^r}(\Phi(X),\Phi(A))<\varepsilon$. To this end, we consider the following regions of $\mathbb{S}^{d-1}$:
    \begin{table}[h]
    \begin{tabular}{ll}
          $\mathcal{R}_1:=\left\{\theta\in\mathbb{S}^{d-1}:\rho_A(\theta)\leq\frac{1}{j} 
          \right\}$ & $\mathcal{R}_3:=\left\{\theta\in\mathbb{S}^{d-1}:j<\rho_A(\theta)<2j+3
          \right\}$ \\
          $\mathcal{R}_2:=\left\{\theta\in\mathbb{S}^{d-1}:\frac{1}{j}<\rho_A(\theta)\leq j 
          \right\}$ & $\mathcal{R}_4:=\left\{\theta\in\mathbb{S}^{d-1}:\rho_A(\theta)\geq 2j+3 
          \right\}$
    \end{tabular}
    \end{table}\\
    Observe that $\mathcal{R}_1$ and $\mathcal{R}_4$ are always non-empty, while $\mathcal{R}_2$ and $\mathcal{R}_3$ may be empty, depending on $A$. 
    
    Consider $\theta\in\mathcal{R}_1$. Since $d_{AW^r}(A,X)<\eta$, then by Lemma \ref{lemm:axu-calcular-rdAW}-(3), we have that  $\delta(A\cap(2j+3)B^d_2,X\cap(2j+3)B^d_2)<\eta$. Hence, 
    $$
    \rho_A(\theta)-\eta<\min\left\{\rho_{X}(\theta),2j+3\right\}<\rho_A(\theta)+\eta<\frac{1}{j}+\frac{1}{2j+3}<j,
    $$
    which yields to $\rho_{X}(\theta)<\frac{1}{j}+\eta$. Consequently, $\rho_{\Phi(X)}(\theta)=\infty$  or $\rho_{\Phi(X)}(\theta)>\frac{j}{1+j\eta}$. For the latter, notice that we actually have $\rho_{\Phi(X)}(\theta)>\frac{j}{1+j\eta}>j-\varepsilon$
    because of
    $\eta<\frac{\varepsilon}{j(j+\varepsilon)}<\frac{\varepsilon}{j(j-\varepsilon)}$. Particularly, for every $\theta\in\mathcal{R}_1$, we  have that
    $$
    \min\left\{\rho_{\Phi(X)}(\theta), j\right\}
    >\frac{j}{1+j\eta}>j-\varepsilon.
    $$
    Since $\min\left\{\rho_{\Phi(A)}(\theta),j\right\}=j$ on $\mathcal{R}_1$, then
    \begin{align}\label{eq:Region1}
           \sup_{\theta\in\mathcal{R}_1}\left|\min\left\{\rho_{\Phi(A)}(\theta),j\right\}-\min\left\{\rho_{\Phi(X)}(\theta), j\right\}\right|
           \leq j-\frac{j}{1+j\eta}
           <\varepsilon.
    \end{align}
    Now, suppose that $\mathcal{R}_2\neq\emptyset$ and let $\theta\in\mathcal{R}_2$ be arbitrary. First, let us suppose that $\rho_X(\theta)\leq\frac{1}{j}$. Then, by Lemma \ref{lemm:axu-calcular-rdAW}-(3), we know that
    \begin{equation}\label{eq:auxprovecontPhi1}
         \left|\rho_{A}(\theta)-\rho_{X}(\theta)\right|=
    \left|\min\left\{\rho_{A}(\theta),2j+3\right\}-\min\left\{\rho_{X}(\theta), 2j+3\right\}\right|<\eta.
    \end{equation}
    Particularly, $\rho_X(\theta)>\rho_A(\theta)-\eta>\frac{1}{j}-\eta>0$. Hence $j\leq\rho_{\Phi(X)}(\theta)<\frac{j}{1-j\eta}$, and
    \begin{align*}
        \left|\min\left\{\rho_{\Phi(A)}(\theta),j\right\}-\min\left\{\rho_{\Phi(X)}(\theta), j\right\}\right|&\leq\left|\rho_{\Phi(A)}(\theta)-\rho_{\Phi(X)}(\theta)\right|\\
        &\leq\eta\rho_{\Phi(A)}(\theta)\rho_{\Phi(X)}(\theta)\\
        &<\frac{j^2\eta }{1-j\eta}.
    \end{align*}
    Since $\eta<\frac{\varepsilon}{j^2+j\varepsilon}$,
    then,
    for every $\theta\in\mathcal{R}_2$ such that $\rho_X(\theta)\leq\frac{1}{j}$, we have that
    $$
    \left|\min\left\{\rho_{\Phi(A)}(\theta),j\right\}-\min\left\{\rho_{\Phi(X)}(\theta), j\right\}\right|<\frac{j^2\eta }{1-j\eta}<\varepsilon.
    $$
    Now, let us suppose that $\theta\in\mathcal{R}_2$ is such that $\frac{1}{j}<\rho_X(\theta)\leq j$. In this case, by Lemma \ref{lemm:axu-calcular-rdAW}-(3), the inequality in (\ref{eq:auxprovecontPhi1}) also holds. 
    Thus,
    \begin{align*}
        \left|\min\left\{\rho_{\Phi(A)}(\theta),j\right\}-\min\left\{\rho_{\Phi(X)}(\theta), j\right\}\right|&=
        \left|\rho_{\Phi(A)}(\theta)-\rho_{\Phi(X)}(\theta) \right|\\
        &<j^2\left|\rho_{A}(\theta)-\rho_{X}(\theta) \right|\\
        &<j^2\eta\\
        &<\varepsilon.
    \end{align*}

    Now, suppose that $\theta\in\mathcal{R}_2$ is such that $\rho_X(\theta)>j$. Notice that $\rho_X(\theta)$ must be finite. Otherwise $\rho_X(\theta)=\infty$ and, from Lemma \ref{lemm:axu-calcular-rdAW}-(3), it follows that
    $$
    \left|\rho_A(\theta)-(2j+3)\right|=\left|\min\left\{\rho_{A}(\theta),(2j+3)\right\}-\min\left\{\rho_{X}(\theta), (2j+3)\right\}\right|<\eta.
    $$
    Hence $\rho_{A}(\theta)>2j+3-\eta>2j+2$, which is a contradiction. Thus  $j<\rho_X(\theta)<\infty$ and, since $\frac{1}{j}<\rho_A(\theta)\leq j$, we have:
    \begin{align*}
        \left|\min\left\{\rho_{\Phi(A)}(\theta),j\right\}-\min\left\{\rho_{\Phi(X)}(\theta), j\right\}\right|&\leq
        \left|\rho_{\Phi(A)}(\theta)-\rho_{\Phi(X)}(\theta) \right|\\
        &<\left|\rho_{A}(\theta)-\rho_{X}(\theta) \right|\\
        &<\eta<\varepsilon,
    \end{align*}
    for every $\theta\in\mathcal{R}_2$ such that $\rho_X(\theta)>j$.  By putting together the previous inequalities, we have just proved that if $\mathcal{R}_2\neq\emptyset$, then
    \begin{equation}\label{eq:Region2}
        \sup_{\theta\in\mathcal{R}_2}\left|\min\left\{\rho_{\Phi(A)}(\theta),j\right\}-\min\left\{\rho_{\Phi(X)}(\theta), j\right\}\right|\leq\max\left\{\frac{j^2\eta }{1-j\eta}, j^2\eta, \eta\right\}<\varepsilon.
    \end{equation}
    Next, assume that $\mathcal{R}_3\neq\emptyset$ and let $\theta\in \mathcal{R}_3$ be arbitrary. If $\theta\in\mathcal{R}_3$ is such that $\rho_X(\theta)\leq j$, then by Lemma \ref{lemm:axu-calcular-rdAW}-(3), the inequality in (\ref{eq:auxprovecontPhi1}) holds. Hence $\left|\rho_{A}(\theta)-\rho_{X}(\theta)\right|<\eta$ and $\rho_X(\theta)>j-\eta$. Thus $\frac{1}{j}\leq\rho_{\Phi(X)}(\theta)<\frac{1}{j-\eta}<j$, and
    \begin{align*}
        \left|\min\left\{\rho_{\Phi(A)}(\theta),j\right\}-\min\left\{\rho_{\Phi(X)}(\theta), j\right\}\right|&=
        \left|\rho_{\Phi(A)}(\theta)-\rho_{\Phi(X)}(\theta) \right|\\
        &<\frac{1}{j(j-\eta)}\left|\rho_{A}(\theta)-\rho_{X}(\theta) \right|\\
        &<\frac{\eta}{j(j-\eta)}\\
        &<\varepsilon.
    \end{align*}
    The last inequality follows from the fact that $\eta<\frac{\varepsilon}{j^2+j\varepsilon}<\frac{j^2\varepsilon}{1+j\varepsilon}$. 

    For $\theta\in\mathcal{R}_3$ such that $j<\rho_X(\theta)<2j+3$, observe that the inequality (\ref{eq:auxprovecontPhi1}) holds again. Therefore,
    \begin{align*}
        \left|\min\left\{\rho_{\Phi(A)}(\theta),j\right\}-\min\left\{\rho_{\Phi(X)}(\theta), j\right\}\right|&=
        \left|\rho_{\Phi(A)}(\theta)-\rho_{\Phi(X)}(\theta) \right|\\
        &<\frac{1}{j^2}\left|\rho_{A}(\theta)-\rho_{X}(\theta) \right|\\
        &<\frac{\eta}{j^2}<\frac{\varepsilon}{(1+j\varepsilon)}.
    \end{align*}
    
    Finally, if $\theta\in\mathcal{R}_3$ and $\rho_{X}(\theta)\geq2j+3$, then again by Lemma \ref{lemm:axu-calcular-rdAW}-(3), we have that:
    $$
        \left|\rho_{A}(\theta)-(2j+3)\right|=\left|\min\left\{\rho_{A}(\theta),2j+3\right\}-\min\left\{\rho_{X}(\theta), (2j+3\right\}\right|<\eta
    $$
    Thus $\rho_{A}(\theta)>2j+3-\eta>2j+2$, which yields to:
    \begin{align*}
        \left|\min\left\{\rho_{\Phi(A)}(\theta),j\right\}-\min\left\{\rho_{\Phi(X)}(\theta), j\right\}\right|&=
        \left|\rho_{\Phi(A)}(\theta)-\rho_{\Phi(X)}(\theta) \right|\\
        &<\frac{1}{2j+2}+\frac{1}{2j+3}\\
        &<\frac{1}{j+1}.
    \end{align*}
    By putting together the previous inequalities, we have shown that if $\mathcal{R}_3\neq\emptyset$, then
    \begin{equation}\label{eq:Region3}
        \sup_{\theta\in\mathcal{R}_3}\left|\min\left\{\rho_{\Phi(A)}(\theta),j\right\}-\min\left\{\rho_{\Phi(X)}(\theta), j\right\}\right|\leq\max\left\{\frac{\eta}{j(j-\eta)}, \frac{\eta}{j^2}, \frac{1}{j+1}\right\}
        <\varepsilon.
    \end{equation}    
    For the last region $\mathcal{R}_4$. Notice that from Lemma \ref{lemm:axu-calcular-rdAW}-(3), it follows that:
    $$
    \min\left\{\rho_X(\theta),2j+3\right\}>\min\left\{\rho_A(\theta),2j+3\right\}-\eta=2j+3-\eta>2j+2. 
    $$
    Hence $\rho_X(\theta)>2j+2$ and,
     \begin{align*}
        \left|\min\left\{\rho_{\Phi(A)}(\theta),j\right\}-\min\left\{\rho_{\Phi(X)}(\theta), j\right\}\right|=
        \left|\rho_{\Phi(A)}(\theta)-\rho_{\Phi(X)}(\theta) \right|
        <\frac{1}{j+1}.
    \end{align*}
    Consequently,
    \begin{equation}\label{eq:Region4}
        \sup_{\theta\in\mathcal{R}_4}\left|\min\left\{\rho_{\Phi(A)}(\theta),j\right\}-\min\left\{\rho_{\Phi(X)}(\theta), j\right\}\right|\leq \frac{1}{j+1}
        <\varepsilon.
    \end{equation}
    From Lemma \ref{lemm:axu-calcular-rdAW}-(4) and inequalities  (\ref{eq:Region1}), (\ref{eq:Region2}), (\ref{eq:Region3}) and (\ref{eq:Region4}), we have that  $\delta\left(\Phi(A)\cap jB_2^n,\Phi(X)\cap jB_2^n\right)<\varepsilon$ which is equivalent to 
    $$d_{AW^r}(\Phi(A),\Phi(X))<\varepsilon.$$  
    This proves the continuity of the duality at each $A\in\mathcal{S}_{rc}^d$ as in (3) of this proof.

    To prove continuity of $\Phi$ at each bounded $A\in\mathcal{S}_{rc}^d$. Let $j_A\geq1$ be the smallest integer such that $A\subseteq j_AB_2^d$ and consider  $\varepsilon_1\in\left(\frac{1}{j+1},\frac{1}{j}\right]$, for some integer $j$ such that $j>j_A$. Observe that if we take 
    $$
    0<\eta_1<\min\left\{\frac{1}{2j+3},\frac{\varepsilon_1}{j^2+j\varepsilon_1} \right\},
    $$ 
    then, by (\ref{eq:Src-bouded})), for each $X\in\mathcal{S}_{rc}^d$ with 
    $d_{AW^r}(A,X)<\eta_1$, $X\subseteq2j_AB_2^d$. Hence,
    $\mathcal{R}_4$ and the corresponding subset 
    $$
    \left\{\theta\in\mathbb{S}^{d-1}:\rho_X(\theta)\geq 2j+3 \right\}
    $$ 
    must be empty. Thus, we can proceed as in the proof of the previous case to show that $\delta\left(\Phi(A)\cap jB_2^2,\Phi(X)\cap jB_2^2\right)<\varepsilon_1$. Consequently $d_{AW^r}(\Phi(A),\Phi(X))<\varepsilon_1$, and $\Phi$ is continuous at each $A\in\mathcal{S}_{rc}^d$ bounded.
    

    In order to prove continuity at each $A\in\mathcal{S}_{rc, (0)}^d$. Let $j'_A\geq1$ be the smallest integer such that $\frac{1}{j'_A}B_2^d\subseteq A$, and consider $\varepsilon_2\in\left(\frac{1}{j+1},\frac{1}{j}\right]$ for some integer $j>j'_A$. Notice that by the Lemma \ref{lem:Src(0)-is-open}, for every $X\in\mathcal{S}_{rc}^d$ such that $d_{AW^r}(A,X)<\frac{1}{2j'_A}$, we have that $\frac{1}{2j'_A}B_2^d\subseteq X.$ Particularly, if we choose 
    $$
    0<\eta_2<\min\left\{\frac{1}{2j+3},\frac{\varepsilon_2}{j^2+j\varepsilon_2} \right\},
    $$ 
    and any $X\in\mathcal{S}_{rc}^d$ with $d_{AW^r}(A,X)<\eta_2$, then $\frac{1}{2j'_A}B_2^d\subseteq X$. Furthermore, the regions $\mathcal{R}_1$ and $\left\{\theta\in\mathbb{S}^{d-1}:\rho_X(\theta)\leq\frac{1}{j}\right\}$ are empty. Now, we can proceed as in the proof of the first case to prove that 
    $$
    \delta\left(\Phi(A)\cap jB_2^2,\Phi(X)\cap jB_2^2\right)<\varepsilon_2.
    $$ 
    Hence,
    $d_{AW^r}(\Phi(A),\Phi(X))<\varepsilon_2$
    which shows continuity at any $A\in\mathcal{S}_{rc(0)}^d$.  This finishes the proof.
\end{proof}

\begin{corollary}\label{cor:duality-Phi-cont-S1}
The star duality  $\Phi:\left(\mathcal{S}_1^d,d_{AW^r}\right)\rightarrow\left(\mathcal{S}_1^d,d_{AW^r}\right)$ is continuous.

\end{corollary}

\subsection{Flowers of convex sets}
For every $x\in\mathbb{R}^d$, let $B_x$ denote the closed ball having as a diameter the segment $[0,x]$. That is, 
$$
B_x:=\left\{y\in\mathbb{R}^n:\left\Vert y-\frac{x}{2}\right\Vert\leq\frac{\|x\|}{2}\right\}.
$$
In \cite[Definition 9.3]{MMilmanRotem}, the family of flowers is introduced. A star body $F\subseteq\mathbb{R}^d$ is called a \textit{flower} if $F=rc\left({\bigcup}_{x\in C}B_x\right)$, for some non-empty $C\subseteq\mathbb{R}^d$ closed. The family of flowers in $\mathbb{R}^d$ is denoted by $\mathcal{F}^d$. 

Given a convex set $K\in\mathcal{K}_0^d$, \textit{ the flower of} $K$, $K^\clubsuit$, is the star body with radial function $\rho_{K^\clubsuit}=h_K$, where $h_K$ denotes the \textit{support function} of $K\in\mathcal{K}_0^d$.
By way of example, for every $x\in\mathbb{R}^d$, $[0,x]^\clubsuit=B_x$. In \cite[Theorem 9.4]{MMilmanRotem}, it is proved that, for every $K\in\mathcal{K}_0^d$, $K^\clubsuit\in\mathcal{F}^d$  and the map 
\begin{align*}
    \clubsuit:&\mathcal{K}_0^d\rightarrow\mathcal{F}^d\\
    &K\rightarrow K^\clubsuit
\end{align*}
is a bijection. Notice that if $K\in\mathcal{K}_{0,b}^d$, then $K^\clubsuit$ is always compact. However, in general, flowers may be non-closed star bodies. Indeed, for each $u\in\mathbb{S}^{d-1}$, let $R_u$ denote the ray $\{tu:t\geq0\}$ and let $H^+_u$ denote the open halfspace $\{x\in\mathbb{R}^d:\langle x,u\rangle>0\}$. Then, the flower $(R_u)^\clubsuit=H_u^+\cup\{0\}$  is not closed. 

Recall that \textit{the polar duality} $\circ:\mathcal{K}_0^d\rightarrow\mathcal{K}_0^d$ is the map sending each $K\in\mathcal{K}_0^d$, to its polar set $K^\circ:=\left\{y\in\mathbb{R}^d:\langle y,x\rangle\leq1,\textit{ for all }x\in K\right\}$. We refer to \cite{ArtsteinMilman2007,ZooDualities,NataliaLuisa2, Schneider2014,Slomka2011} for the basic properties and recent works regarding this duality.

A fundamental relation between flowers of convex sets, the polar duality, and the duality $\Phi$ is established in \cite[Theorem 9.12]{MMilmanRotem}. There it is proved that $\circ$
can be written as the composition
\begin{equation}\label{eq:conmut-diagrama}
\circ:\mathcal{K}_0^d\overset{\clubsuit}{\longrightarrow}\mathcal{F}^d\overset{\Phi}{\longrightarrow}
\mathcal{K}_0^d.
\end{equation}
That is, $K^\circ=\Phi(K^\clubsuit)$ for every $K\in\mathcal{K}_0^d$.
 
In the specific cases of compact flowers and convex bodies, we can further explore the properties of the decomposition (\ref{eq:conmut-diagrama}). To do so, let us denote by $\mathcal{F}_{b}^d$
the family of bounded flowers and, by $\mathcal{F}_{(0),b}^d$,
the family of bounded flowers that contain the origin in their interior. Then the following holds,
\begin{remark} \label{rem:isometria-flores}The flower map $\clubsuit_{|}:\left(\mathcal{K}_{0,b}^d,d_H\right)\rightarrow\left(\mathcal{F}_{b}^d,\delta\right)$ is an isometric isomorphism.
\end{remark}
\begin{proof}
    By \cite[Theorem 9.4]{MMilmanRotem}, the map $\clubsuit:\mathcal{K}_{0}^d\rightarrow\mathcal{F}^d$ is a bijection. Particularly, $\clubsuit\left(\mathcal{K}_{0,b}^d\right)=\mathcal{F}_b^d$ and, for every $A,K\in\mathcal{K}_{0,b}^d$, we have that
    \begin{align*}
        d_H(A,K)=\left\| h_A-h_K\right\|_\infty=
        \left\| \rho_{A^\clubsuit}-\rho_{K^\clubsuit}\right\|_\infty=\delta(A^\clubsuit,K^\clubsuit),
    \end{align*}
    which proves the result.
\end{proof}

By combining decomposition (\ref{eq:conmut-diagrama}) with Remarks \ref{rem:dualidad-continua-acotados}, \ref{rem:d_AWr de convexos}-(3), and \ref{rem:isometria-flores}, along with the continuity of the polar duality on $(\mathcal{K}_{(0),b}^d, d_H)$, we conclude that the following diagram commutes,
\begin{center}
\begin{tikzpicture} 
\node at (-2.5, 0){$\left(\mathcal{K}_{(0),b}^d,d_H\right)$}; \node at (2.5, 0){$\left(\mathcal{K}_{(0),b}^d,d_H\right)$}; \draw[->, thick] (-1.2,0) -- node[above] {$\circ$} ++(2.4,0); \node at (0, -1.85) {$\left(\mathcal{F}_{(0),b}^d,\delta\right)$};\draw[->, thick] (-1.3,-0.3) -- node[below, pos=0.4] {$\clubsuit$} ++(1.1,-1); \draw[->, thick] (0.2,-1.3) -- node[below, pos=0.6, xshift=0.8cm] {$\Phi_{|{\mathcal{F}_{(0),b}^d}}=\hat{\Phi}$} ++(1.1,1.1); 
\draw [thick, right hook ->](0,-2.4) -- node[below, yshift=0.3cm, xshift=-0.2cm]{$\iota$} (0, -3.5) ; \node at (0, -3.9){$\left(\mathcal{S}_{(0),b}^d,\delta\right)$}; 
\end{tikzpicture}
\end{center}
Here 
$\iota:\mathcal{F}_{(0),b}^d\hookrightarrow\mathcal{S}_{(0),b}^d$ denotes the inclusion map. Consequently, in the compact case, the decomposition (\ref{eq:conmut-diagrama}) remains valid when the usual metric topologies are considered. 


As we will see below, analyzing the properties of decomposition (\ref{eq:conmut-diagrama}) becomes more intricate when the metric topologies $d_{AW}$ and $d_{AW^r}$ on  $\mathcal{K}_0^d$ and $\mathcal{S}_{rc}^d$, respectively, are considered.



\begin{proposition}\label{prop:flore-discontinuous-dAWr}\begin{enumerate}
    \item The spaces $\left(\mathcal{K}_{0}^d,d_{AW^r}\right)$ and $\left(\mathcal{F}^d,d_{AW^r}\right)$ are homeomorphic.
    \item The flower map $\clubsuit:\left(\mathcal{K}_{0}^d,d_{AW}\right)\rightarrow\left(\mathcal{F}^d,d_{AW^r}\right)$ is not continuous.
\end{enumerate}
     
\end{proposition}
\begin{proof}
    (1) It follows directly from Theorem \ref{thm:duality-awr-continua}, and the fact that  
    $\Phi\left(\mathcal{F}^d\right)=\mathcal{K}_0^d$ (see \cite[Theorem 9.4]{MMilmanRotem}).

    (2) Let $K,K_n\subseteq\mathbb{R}^2$, $n\in\mathbb{N}$, be defined as $K=[0,1]\times[0,\infty)$ and
    $$
    K_n=K\cup\left\{(x,y)\in\mathbb{R}^2:x\geq1\text{ and } y\geq n(x-1)\right\}.
    $$
    We shall prove that the sequence $(K_n)_n$ converges in $d_{AW}$ to $K$, but $(K^\clubsuit_n)_n$ does not converges to $K^\clubsuit$ in $d_{AW^r}$. The result for an arbitrary dimension will be obtained by considering the embedding $i:\mathbb{R}^2\rightarrow\mathbb{R}^d$, given by $i(x,y)=(x,y,0,\ldots,0)$,  the convex set $i(K)$, and the sequence $(i(K_n))_n$. 

    To prove the $d_{AW}$-convergence of $(K_n)_n$ to $K$, observe that the sequence of distance functionals $(d(\cdot,K_n))_n$ converges pointwise to $d(\cdot,K)$. Since $\tau_W$ and the metric topology induced by $d_{AW}$ agree on $\mathcal{K}_0^d$ (see e.g. \cite[Theorem 3.1.4]{Beer1993}), then $(K_n)_n$ is $d_{AW}$-convergent to $K$. 
    
    On the other hand, notice that $\rho_{K^\clubsuit}(e_1)=h_K(e_1)=1$. Also, for any $t\geq1$ and $n\geq1$, $(t,nt)\in K_n$. Hence
    $\rho_{K_{n}^\clubsuit}(e_1)\geq t$, which yields to  $\rho_{K_{n}^\clubsuit}(e_1)=\infty$ for all $n$. The latter, in combination with Theorem \ref{thm:caract-rWijsmanTopology}, proves that $(K_{n}^\clubsuit)_n$ does not converge in $\tau_{W^r}$ to $K^\clubsuit$. By Remark \ref{rem:dAWR-stronger-than-Wr}, the sequence $(K_{n}^\clubsuit)_n$ is not $d_{AW^r}$-convergent to $K^\clubsuit$ either. 

    To prove the discontinuity of $\clubsuit:\left(\mathcal{K}_{0}^d,d_{AW}\right)\rightarrow\left(\mathcal{F}^d,d_{AW^r}\right)$ for $d>2$. Notice that $\rho_{(i(K))^\clubsuit}(e_1)=1$ and that $\rho_{(i(K_n))^\clubsuit}(e_1)=\infty$ for all $n$. Moreover, for every integer $j\geq1$, we have that:
    \begin{align*}
    d_H\left(i(K)\cap jB_2^d,i(K_n)\cap jB_2^d\right)&=d_H\left(i\left(K\cap jB_2^2\right),i\left(K_n\cap jB_2^2\right)\right)\\
    &\leq d_H(K\cap jB_2^2,K_n\cap jB_2^2).
    \end{align*}
    Thus,  $d_{AW}\left(i(K),i(K_n)\right)\leq d_{AW}(K,K_n).$
    Therefore $(i(K_n))_n$ converges, in $d_{AW}$, to $i(K)$, while $\left((i(K_n))^\clubsuit\right)_n$ does not converges to $(i(K))^\clubsuit$ in $d_{AW^r}$. This finishes the proof.     
\end{proof}


At this stage, it is natural to consider an alternative metric on the family of flowers, namely the one induced on $\mathcal{F}^d$ by the map $\clubsuit$ and the space $\left(\mathcal{K}_0^d,d_{AW}\right)$. For every pair of flowers $F_1=K^\clubsuit_1$ and $F_2=K^\clubsuit_2$ with $K_1,K_2 \in\mathcal{K}_0^d$, the distance $d_{\clubsuit}(F_1,F_2)$ is defined by:
\begin{equation}\label{eq:flower-distance}
   d_{\clubsuit}(F_1,F_2):=d_{AW}(K_1,K_2). 
\end{equation}

In this case, the flower map $\clubsuit:\left(\mathcal{K}_0^d,d_{AW}\right)\rightarrow\left(\mathcal{F}^d,d_{\clubsuit}\right)$ is clearly an isometric isomorphism. However, as shown in the proof of Proposition \ref{prop:flore-discontinuous-dAWr}, the identity map 
$$
I:\left(\mathcal{F}^d,d_{\clubsuit}\right)\rightarrow\left(\mathcal{F}^d,d_{AW^r}\right)
$$
is discontinuous. This proves that on $\mathcal{F}^d$,
the metric topologies determined by $d_{AW^r}$ and $d_{\clubsuit}$  are different.

In this context, the following natural questions arise:

\textbf{Question 1:} Is it possible to establish an intrinsic formulation of the distance $d_{\clubsuit}$ on the family $\mathcal{F}^d$, analogous to the intrinsic characterizations of the Hausdorff distance and the metric $\delta$ for bounded flowers?

\textbf{Question 2:} If such a formulation exists, one may ask whether the metric topology extends to a larger class of star bodies, and to investigate the topological properties of the duality $\Phi$ within this broader framework.

Recall that the Hilbert cube $Q$ is the topological product $\Pi_{i=1}^\infty[-1,1]$. In \cite[Theorem 3.4]{NataliaLuisa2}, tools from the theory of Hilbert-cube manifolds were used to show that the space $(\mathcal{K}_0^d,d_{AW})$ is homeomorphic with $Q$. Furthermore, combining \cite[Corollary 4]{Slomka2011} and \cite[Theorem 2]{NataliaLuisa2}, it follows that any  duality $g$ on $\mathcal{K}_0^d$ is automatically continuous with respect to $d_{AW^r}$, and
is topologically conjugate\footnote{Let $X$ and $Y$ be topological spaces. Then, two continuous functions $f,h:X\rightarrow Y$ are called \textit{topologically conjugate} if there is a homeomorphism $\Psi:X\rightarrow Y$ such that $g\circ\Psi=\Psi\circ f$ for all $x\in X$.} to the polar duality if and only if it has a unique fixed point.

In the setting of star bodies, the map $\Phi:\mathcal{S}_{rc}^d\rightarrow\mathcal{S}_{rc}^d$, as well as its restriction $\Phi_{|\mathcal{S}_1^d}$, are dualities with the Euclidean ball $B_2^d$ as their unique fixed point. Moreover, Theorem \ref{thm:duality-awr-continua} and Corollary \ref{cor:duality-Phi-cont-S1}, show that both are continuous with respect to $d_{AW^r}$. These facts naturally lead to the following question:

\textbf{Question 3:} What are the homeomorphism types (with respect to $d_{AW^r}$) of  $\mathcal{S}_{rc}^d$ and $\mathcal{S}_1^d$. Is there a topological characterization of the dualities $\Phi$ or $\Phi_{|\mathcal{S}_1^d}$, analogous to that of the polar duality on $(\mathcal{K}_0^d,d_{AW})$?

\bibliographystyle{acm}

\begin{thebibliography}{10}

\bibitem{ArtsteinMilman2007}
{\sc {A}rtstein{-A}vidan, S., and {M}ilman, V.}
\newblock {A} characterization of the concept of duality.
\newblock {\em Electron. Res. Announc. Math. Sci. 14\/} (2007), 42--59.

\bibitem{ArtsteinMilman2008}
{\sc {A}rtstein {A}vidan, S., and {M}ilman, V.}
\newblock {T}he concept of duality for measure projections of convex bodies.
\newblock {\em J. Funct. Anal. 254}, 10 (2008), 2648--2666.

\bibitem{ZooDualities}
{\sc {A}rtstein {A}vidan, S., {S}adovsky, S., and {W}yczesany, K.}
\newblock A zoo of dualities.
\newblock {\em {J}. {G}eom. {A}nal. 33}, 238 (2023), 1--40.

\bibitem{Beer1975}
{\sc Beer, G.}
\newblock {S}tarshaped sets and the {H}ausdorff metric.
\newblock {\em Pacific J. Math. 61}, 1 (1975), 21--27.

\bibitem{Beer1993}
{\sc Beer, G.}
\newblock {\em {T}opologies on {C}losed and {C}losed {C}onvex {S}ets}.
\newblock {K}luwer {A}cademic {P}ublishers {G}roup, {D}ordrecht, 1993.

\bibitem{BeerKlee1987}
{\sc Beer, G., and Klee, V.}
\newblock {L}imits of starshaped sets.
\newblock {\em Arch. Math. (Basel) 48}, 3 (1987), 241--249.

\bibitem{Dugundji66}
{\sc Dugundji, J.}
\newblock {\em {T}opology}.
\newblock {A}llyn and {B}acon, Boston, 1966.

\bibitem{GardnerBook}
{\sc {G}ardner, R.}
\newblock {\em {G}eometric {T}omography}, second~ed.
\newblock {C}ambridge {U}niversity {P}ress, {N}ew {Y}ork, 2006.

\bibitem{Moszynska2020}
{\sc {H}ansen, G., {H}erburt, I., {M}artini, H., and {M}oszy\'{n}ska, M.}
\newblock {S}tarshaped {s}ets.
\newblock {\em Aequationes Math. 94}, 6 (2020), 1001--1092.

\bibitem{HansenMartini2010}
{\sc Hansen, G., and Martini, H.}
\newblock On closed starshaped sets.
\newblock {\em J. Convex Anal. 17}, 2 (2010), 659--671.

\bibitem{HansenMartini2011}
{\sc Hansen, G., and Martini, H.}
\newblock {S}tarshapedness vs. convexity.
\newblock {\em Results Math. 59}, 1-2 (2011), 185--197.

\bibitem{NataliaLuisa2}
{\sc {H}igueras {M}onta{\~n}o, L., and {J}onard{-}{P}{\'e}rez, N.}
\newblock {A} topological insight into the polar involution of convex sets.
\newblock {\em Israel J. Math. 263\/} (2024), 437--474.

\bibitem{Klain96}
{\sc Klain, D.}
\newblock {S}tar valuations and dual mixed volumes.
\newblock {\em Adv. Math. 121}, 1 (1996), 80--101.

\bibitem{LipschitzStarbodies}
{\sc Lin, Y., and Wu, Y.}
\newblock Lipschitz star bodies.
\newblock {\em Acta Math. Sci. Ser. B (Engl. Ed.) 43}, 2 (2023), 597--607.

\bibitem{Lutwak88}
{\sc {L}utwak, E.}
\newblock Intersection bodies and dual mixed volumes.
\newblock {\em Adv. in Math. 71}, 2 (1988), 232--261.

\bibitem{MMilmanRotem}
{\sc {M}ilman, E., {M}ilman, V., and {R}otem, L.}
\newblock {R}eciprocals and flowers in convexity.
\newblock In {\em {G}eometric {A}spects of {F}unctional {A}nalysis. {V}ol. {II}}. Springer, Cham, 2020.

\bibitem{MilmanRotem2017RadialSums}
{\sc Milman, V., and Rotem, L.}
\newblock Characterizing the radial sum for star bodies.
\newblock In {\em {G}eometric {A}spects of {F}unctional {A}nalysis}. Springer, Cham, 2017, pp.~319--329.

\bibitem{MilmanRotem2020}
{\sc {M}ilman, V., and {R}otem, L.}
\newblock {N}ovel view on classical convexity theory.
\newblock {\em J. Math. Phys. Anal. Geom. 16}, 3 (2020), 291--311.

\bibitem{Moszynska1999}
{\sc {M}oszy\'{n}ska, M.}
\newblock {Q}uotient star bodies, intersection bodies, and star duality.
\newblock {\em J. Math. Anal. Appl. 232}, 1 (1999), 45--60.

\bibitem{MoszynskaBook}
{\sc {M}oszy\'{n}ska, M.}
\newblock {\em {S}elected {T}opics in {C}onvex {G}eometry}.
\newblock Birkh\"{a}user Boston, Inc., Boston, MA, 2006.
\newblock {T}ranslated and revised from the 2001 {P}olish original.

\bibitem{Nagatabook}
{\sc Nagata, J.}
\newblock {\em {M}odern {G}eneral {T}opology}, second~ed.
\newblock {N}orth-{H}olland {P}ublishing {C}o., {A}msterdam, 1985.

\bibitem{RockWets1998}
{\sc Rockafellar, R.~T., and Wets, R. J.-B.}
\newblock {\em {V}ariational {A}nalysis}.
\newblock {S}pringer-{V}erlag, Berlin, 1998.

\bibitem{SakaiYaguchi2006}
{\sc Sakai, K., and Yaguchi, M.}
\newblock The {AR}-property of the spaces of closed convex sets.
\newblock {\em Colloq. Math. 106}, 1 (2006), 15--24.

\bibitem{Schneider2014}
{\sc {S}chneider, R.}
\newblock {\em Convex {B}odies: {T}he {B}runn-{M}inkowski {T}heory, {S}econd {E}xpanded {E}dition}.
\newblock Cambridge University Press, Cambridge, 2014.

\bibitem{Slomka2011}
{\sc {S}lomka, B.}
\newblock {O}n duality and endomorphisms of lattices of closed convex sets.
\newblock {\em {A}dv. {G}eom. 11}, 2 (2011), 225--239.

\bibitem{Sojka2013}
{\sc S\'ojka, G.}
\newblock Metrics in the family of star bodies.
\newblock {\em {A}dv. {G}eom. 13}, 1 (2013), 117--144.

\bibitem{Wijsman1966}
{\sc {W}ijsman, R.}
\newblock {C}onvergence of sequences of convex sets, cones and functions. {II}.
\newblock {\em Trans. Amer. Math. Soc. 123\/} (1966), 32--45.

\end{thebibliography}

\end{document}